\documentclass[3p,preprint]{elsarticle}

\usepackage{amsfonts}
\usepackage{amsmath}
\usepackage{geometry}
\usepackage{csquotes}
\usepackage{algorithmicx}
\usepackage{algorithm}
\usepackage{algpseudocode}
\usepackage{color}

\usepackage{empheq} 
\usepackage{physics}
\usepackage{comment}

\newtheorem{theorem}{Theorem}[section]

\newtheorem{corollary}[theorem]{Corollary}

\newtheorem{example}[theorem]{Example}

\newtheorem{lemma}[theorem]{Lemma}

\newtheorem{remark}[theorem]{Remark}

\newenvironment{proof}[1][Proof]{\noindent\textbf{#1.} }{\ \rule{0.5em}{0.5em}}

\biboptions{sort&compress}

\begin{document}

\bigskip

\bigskip 
\begin{frontmatter}

\title{Weighted and unweighted enrichment strategies for solving the Poisson problem with Dirichlet boundary conditions}

\author[address-CS]{Francesco Dell'Accio}
\ead{francesco.dellaccio@unical.it}

\author[address-Messina,PELO]{Luca Desiderio}
\ead{ldesiderio@unime.it}

\author[address-Pau]{Allal Guessab}
\ead{guessaballal7@gmail.com}

\author[address-CS]{Federico Nudo\corref{corrauthor}}\cortext[corrauthor]{Corresponding author}
\ead{federico.nudo@unical.it}

\address[address-CS]{Department of Mathematics and Computer Science,\\ University of Calabria, Rende (CS), Italy}

 \address[address-Messina]{Department of Mathematical and Computer Sciences, Physical Sciences and Earth Sciences,\\ University of Messina, Italy}

\address[PELO]{Member of Accademia Peloritana dei Pericolanti, Messina, Italy.}

\address[address-Pau]{Avenue Al Walae, 73000, Dakhla, Morocco}

\begin{abstract}
In this paper, we propose weighted and unweighted enrichment strategies to enhance the accuracy of the linear lagrangian finite element for solving the Poisson problem with Dirichlet boundary conditions. We first recall key examples of admissible enrichment functions, specifically designed to overcome the limitations of the linear lagrangian finite element in capturing solution features such as sharp gradients and boundary-layer phenomena. We then introduce two novel three-parameter families of weighted enrichment functions and derive an explicit error bound in $L^2$-norm.  Numerical experiments confirm the effectiveness of the proposed approach in improving approximation accuracy, demonstrating its potential for a wide range of applications.
\end{abstract}


\end{frontmatter}

\section{Introduction}
The finite element method (FEM) is a cornerstone of numerical analysis, widely used for numerically solving partial differential equations (PDEs) that arise in fields such as engineering, physics and applied mathematics. Its widespread adoption stems from its versatility in handling complex geometries and boundary conditions, making it a powerful tool for modeling various physical phenomena. Among fundamental PDEs, the Poisson equation plays a central role, as it describes key processes in areas such as heat conduction, electrostatics and fluid flow. Consequently, the development of accurate and efficient numerical methods for solving the Poisson problem remains a topic of significant interest~\cite{lin2005superconvergence, agbezuge2006finite, vartziotis2013improving, zhang2016robustness, wang2016efficient}.

\noindent
In the finite element framework, as well as in its modern extensions, such as the virtual element method and the discontinuous Galerkin method, which has been recently coupled with boundary element methods, see~\cite{desiderio2022CVEM-BEM, mascotto2020fem, desiderio2021virtual, desiderio2022coupling, erath2022mortar, desiderio2023cvem},
the computational domain is partitioned into a set of polytopes, and the solution of the PDE is locally approximated within each element. The global approximation is a piecewise function obtained by assembling these local approximations. If the global approximation exhibits discontinuities at the interfaces between elements, the finite element is said to be \textit{nonconforming}; otherwise, it is classified as \textit{conforming}~\cite{Shi:2013:NQF, Fix:1973:OTU}.  One of the simplest and most commonly used finite element in two dimensions is the linear lagrangian finite element,  which locally approximates the solution of the PDE using the unique linear polynomial that interpolates the solution at the vertices of the triangle. However, despite its simplicity and computational efficiency, this approach often struggles to achieve high accuracy for challenging problems, particularly when the solution exhibits wild oscillations or requires high resolution. To overcome these limitations, enrichment strategies have been developed to improve the approximation capabilities of finite elements. These approaches employ additional functions and functionals, referred to as \textit{enrichment functions} and \textit{enriched linear functionals}, respectively~\cite{Guessab2016AADM, Guessab2016RM, Guessab:2017:AUA, DellAccio:2022:AUE, DellAccio2023AGC, DellAccio:2022:ESF, DellAccio:2022:QFA, DellAccio2023nuovo, DellAccioCANWA, nudosolo, nudosolo2}. These enrichment strategies improve the approximation capability within each subdomain, thereby improving the overall accuracy of the finite element solution.

\noindent
In this paper, we build upon the general enrichment strategies introduced in~\cite{DellAccio2023AGC} and apply them to the numerical solution of the Poisson problem with Dirichlet boundary conditions. Our main contribution is the introduction of two novel three-parameter families of weighted enrichment functions, which offers enhanced adaptability and flexibility compared to existing methods. The use of weighted functions in finite element methods has proven effective in various applications~\cite{hammad2020exponential, rukavishnikov2021weighted} and has recently been adopted in the enrichment process for function reconstruction~\cite{nudo5082357general, DellAccio2025new} to improve both flexibility and accuracy in approximation.

\noindent
The paper is organized as follows. In Section~\ref{sec2}, we introduce the strong and weak formulations of the Poisson problem with Dirichlet boundary conditions, establishing the mathematical framework for our analysis. In Section~\ref{sec3}, we recall the general enrichment strategies introduced in~\cite{DellAccio2023AGC}, detailing the construction of the enriched finite element space and the selection of appropriate enrichment functions. Moreover, we introduce two novel three-parameter families of weighted enrichment functions and, for one of these families, we derive an explicit error bound in $L^2$-norm. Finally, in Section~\ref{sec4}, we perform numerical experiments that validate the effectiveness of the proposed approach by comparing it with the linear lagrangian finite element.

\section{Strong and weak forms of Laplace problem}
\label{sec2}
In the affine space associated with $\mathbf{R}^{2}$ and equipped with a fixed orthonormal Cartesian coordinates system $\mathbf{x}=\left(x_{1},x_{2}\right)^{\top}$, let $\Omega$ be an open bounded domain admitting a Lipschitz continuous boundary $\Gamma=\partial\Omega$. We consider the following Boundary Value Problem (BVP) for the Poisson equation
\begin{subequations}\label{dirichlet_problem}
  \begin{empheq}[left=\empheqlbrace]{align}
    \label{dirichlet_problem_1} &-\Delta u(\textbf{x})=f(\textbf{x})& &\textbf{x}\in \Omega\\
    \label{dirichlet_problem_2} &u(\textbf{x})=0& 	 					 &\textbf{x}\in\Gamma
  \end{empheq}
\end{subequations}
\noindent
where $u$ is a scalar unknown function, while $f$ is a given source term. In problem~\eqref{dirichlet_problem}, a homogeneous Dirichlet type boundary condition is chosen for the sake of simplicity, but the case of different and more general boundary data can be handled. Before deriving the weak formulation of problem~\eqref{dirichlet_problem}, we denote by $V$ the subset of the Sobolev space $\text{H}^{1}(\Omega)$ that contains the functions vanishing on $\Gamma$, i.e.
\begin{displaymath}
V:=\text{H}^{1}_{0}(\Omega)=\left\{v\in\text{H}^{1}(\Omega) \, | \,  v_{|_{\Gamma}}=0\right\}.
\end{displaymath}
\noindent
Assuming that $f\in\text{L}^{2}(\Omega)$, we call weak solution of the BVP \eqref{dirichlet_problem} a function $u\in V$ which satisfies the integral identity
\begin{equation}\label{weak_problem}
\int\limits_{\Omega}\nabla u(\mathbf{x})\nabla v(\mathbf{x}) d \mathbf{x}=\int\limits_{\Omega}f(\mathbf{x})v(\mathbf{x}) d \mathbf{x}, \quad \forall v\in V.
\end{equation}
\noindent
The left hand side of the above equation defines the bilinear form 
\begin{equation}\label{bilinear_form_definition}
a(u,v)=\int\limits_{\Omega}\nabla u(\mathbf{x})\nabla v(\mathbf{x})d \mathbf{x}, \quad \forall u,v\in V
\end{equation}
\noindent
while the right hand side introduces the linear functional
\begin{equation}\label{functional_definition}
F(v)=\int\limits_{\Omega}f(\mathbf{x})v(\mathbf{x})d \mathbf{x}, \quad \forall v\in V.
\end{equation}
\noindent
We point out that in \eqref{weak_problem} the Dirichlet boundary condition \eqref{dirichlet_problem_2} is incorporated in the definition of the functional space $V$. Furthermore, we remark that the exact solution of~\eqref{weak_problem} exists and it is unique. This result is a consequence of the $\alpha$-coercivity, with $\alpha=\left(1+C_{\Omega}^{2}\right)^{-1}>0$ (where $C_{\Omega}$ is the Poincaré constant) and $\gamma$-continuity, with $\gamma=1>0$, of the bilinear form $a(\cdot,\cdot)$ and the Lax–Milgram lemma (see~\cite{Ciarlet:2002:TFE}). Additionally, it holds
\begin{equation}
\|u\|_{V}\leq\|F\|_{V^{\star}},
\end{equation}
\noindent
where $V^{\star}$ denotes the dual space of $V$.

\section{Enriched finite element formulation}\label{sec3}
\noindent
In this section, we propose a new enriched finite element Galerkin approach applied to the weak form of the continuous problem in equation \eqref{weak_problem}. Since the starting point of our methodology is the standard FEM, we briefly synthesize its building steps. Once the domain $\Omega$ has been decomposed into a polygonal mesh $\mathcal{T}_{h}(\Omega)$, for any $E\in \mathcal{T}_{h}(\Omega)$, we consider the Ciarlet triplet 
$$\left(E,W_{h}(E),\Sigma_{E}\right),$$ where $W_{h}(E)=\operatorname{span}\left\{\lambda_1^{E},\dots,\lambda_n^{E}\right\}$ is a local space of real-valued functions defined on $E$, and $\Sigma_{E}=\left\{\mathcal{L}_1,\dots,\mathcal{L}_n\right\}$ is a local set of linear functionals on $W_{h}(E)$ such that 
$W_h(E)$ is $\Sigma_E$-unisolvent. This means that for any $i=1,\dots,n$, there exists a unique element $w_{h,i}\in W_h$ satisfying
\begin{equation*}
\mathcal{L}_j\left(w_{h,i}\right)=\delta_{ij},
\end{equation*}
where $\delta_{ij}$ is the Kronecker delta function.
Thus, the Ciarlet triplet uniquely identifies the degrees of freedom associated with $E$. Based on this so-called finite element, we introduce the global space $W_{h}\subset V$ by gluing the local spaces over all the mesh elements, i.e. 
\begin{displaymath}
W_{h}:=\left\{w_{h}\in V \, | \, w_{h_{\mkern 1mu \vrule height 2ex\mkern2mu E}}\in W_{h}(E) \quad \forall E\in\mathcal{T}_{h}(\Omega)\right\}.
\end{displaymath}
\noindent
At this stage, from the weak form in \eqref{weak_problem}, we can formulate the finite element Galerkin problem, which consists of finding $u_{h}\in W_{h}$ such that
\begin{equation}\label{galerkin_weak_problem}
a\left(u_{h},v_{h}\right)=F\left(v_{h}\right), \quad \forall v_{h}\in W_{h}\subset V.
\end{equation}
\noindent
Denoting by $N$ the total number of degrees of freedom of the finite element space $W_{h}$, i.e. its dimension, we introduce its basis functions $\left\{\phi_{1},\phi_{2},\ldots,\phi_{N}\right\}$. Expanding the trial function $u_{h}\in W_{h}$ and the test function $v_{h}\in W_{h}$ in terms of the basis functions, we substitute them into~\eqref{galerkin_weak_problem}, in order to obtain the following linear system 
\begin{equation}\label{linear_system}
\mathbb{K}\boldsymbol\alpha=\mathbf{f} \quad \text{with} \quad \mathbb{K}_{ij}=a\left(\phi_{i},\phi_{j}\right) \quad \text{and} \quad f_{i}=F\left(\phi_{i}\right) \quad i,j=1,\ldots,N,
\end{equation}
\noindent
where $\boldsymbol\alpha\in\mathbf{R}^{N}$ is the unknown coefficient vector, $\mathbb{K}\in\mathbf{R}^{N\times N}$ is the stiffness matrix and $\mathbf{f}\in\mathbf{R}^{N}$ is the external force (load) vector. Then, by solving the linear system~\eqref{linear_system}, we obtain an approximation of the exact solution within the global finite element space $W_{h}$. However, it is well established in the FEM literature that, if the exact solution of the BVP~\eqref{dirichlet_problem} contains complex features, such as singularities, oscillations, exponential decay or boundary-layer phenomena, the accuracy of the standard FEM approaches could be compromised. For this reason, it is beneficial to enrich the Ciarlet triplet with additional enrichment functions, which carry information about the behavior of the exact solution and are introduced through the inclusion of enriched linear functionals.
\subsection{Mesh definition and regularity assumptions}
\noindent
Before detailing the enrichment process, we present such mesh regularity assumptions to characterize the geometry of the elements of $\mathcal{T}_{h}(\Omega)$. In particular, throughout this paper, we consider unstructured meshes, that cover the domain $\Omega$ with a finite number $M_{h}$ of nonoverlapping triangular elements. For each $E\in\mathcal{T}_{h}(\Omega)$, we denote its area by $\left\lvert E\right \rvert$ and its diameter by
\begin{equation*}
h_{{E}}:=\sup_{\mathbf{x},\mathbf{y}\in E}\left\lVert\mathbf{x}-\mathbf{y}\right\rVert,    
\end{equation*}
and we set 
\begin{equation*}
    h=\max_{E\in \mathcal{T}_h(\Omega)} h_E,
\end{equation*}
where $\left\|{\cdot}\right\|$ is the standard Euclidean norm in $\mathbf{R}^2$.  Moreover, we let $\varrho_{{E}}$
represent the diameter of the circle inscribed in $E$, also referred to as the \textit{sphericity} of $E$. In order to immediately exclude very deformed (i.e. stretched) triangles, we require that $\mathcal{T}_{h}(\Omega)$ is regular in the sense that for a suitable $\delta>0$ it holds
\begin{displaymath}
\frac{h_{{E}}}{\varrho_{{E}}}\leq\delta, \quad \forall E\in\mathcal{T}_{h}(\Omega).
\end{displaymath}
\noindent
For convenience, we assume that the vertices of each triangular element $E\in\mathcal{T}_{h}(\Omega)$ are labeled in counter-clockwise order and they are denoted by 
\begin{equation*}
\mathbf{v}_{i}^{E}=\left(v_{i,1}^{E},v_{i,2}^{E}\right)^{\top}, \quad   i=1,2,3.  
\end{equation*}
Furthermore, we use the notation $\mathbf{e}_{i}^{E}$ to indicate the edge of $E$ opposite to the vertex $\mathbf{v}_{i}^{E}$, $i=1,2,3$. Thus, the lengths of the sides of $E$ are defined as
\begin{displaymath}
  \left\lvert \mathbf{e}_1^{E}\right\rvert =\left\| {\mathbf{v}}_{2}^{E}-{\mathbf{v}}_{3}^{E} \right\|, \quad
\left\lvert \mathbf{e}_2^{E}\right\rvert = \left\|{\mathbf{v}}_{1}^{E}-{\mathbf{v}}_{3}^{E}\right\|,  \quad
\left\lvert \mathbf{e}_3^{E}\right\rvert = \left\|{\mathbf{v}}_{1}^{E}-{\mathbf{v}}_{2}^{E}\right\|.
\end{displaymath}
\subsection{The enrichment process}
\noindent
For any triangle $E\in\mathcal{T}_h(\Omega)$, we consider a finite element 
\begin{equation*}
\mathcal{S}=\left(E,W_{h}(E),\Sigma_{E}\right)    
\end{equation*}
and an enrichment finite element 
\begin{equation*}
\widetilde{\mathcal{S}}=\left(E,\widetilde{W}_{h}(E),\widetilde{\Sigma}_{E}\right),
\end{equation*} 
where 
\begin{equation*}
    \Sigma_E=\left\{\mathcal{L}_r:W_{h}(E)\rightarrow\mathbf{R} \, | \, r=1,\dots,n\right\}, \quad  \widetilde{\Sigma}_E=\left\{\widetilde{\mathcal{L}}_s:\widetilde{W}_{h}(E)\rightarrow\mathbf{R} \, | \, s=1,\dots,\widetilde{n}\right\}
\end{equation*}
with 
\begin{equation*}
n:=\text{dim}\left(W_{h}(E)\right), \quad \widetilde{n}:=\text{dim}\left(\widetilde{W}_{h}(E)\right), \quad W_{h}(E)\cap\widetilde{W}_{h}(E)=\left\{\mathbf{0}\right\}.  
\end{equation*}
 We introduce the sets of linearly independent functions $\left\{\lambda^{E}_{1},\lambda^{E}_{2},\ldots,\lambda^{E}_{n}\right\}$ and $\left\{\widetilde{\lambda}^{E}_{1},\widetilde{\lambda}^{E}_{2},\ldots,\widetilde{\lambda}^{E}_{\widetilde{n}}\right\}$ satisfying
 \begin{equation*}
W_{h}(E)=\operatorname{span}\left\{\lambda^{E}_{1},\lambda^{E}_{2},\ldots,\lambda^{E}_{n}\right\}, \quad \widetilde{W}_{h}(E)=\operatorname{span}\left\{\widetilde{\lambda}^{E}_{1},\widetilde{\lambda}^{E}_{2},\ldots,\widetilde{\lambda}^{E}_{\widetilde{n}}\right\}
\end{equation*}
and
\begin{align*}
&\mathcal{L}_{r}\left(\lambda_{i}^{E}\right)=\delta_{ir} \quad \text{with} \quad \mathcal{L}_{r}\in\Sigma_{E} \quad \forall i,r=1,\ldots,n,\\
&\widetilde{\mathcal{L}}_{s}\left(\widetilde{\lambda}_{j}^{E}\right)=\delta_{js} \quad \text{with} \quad \widetilde{\mathcal{L}}_{s}\in\widetilde{\Sigma}_{E} \quad \forall j,s=1,\ldots,\widetilde{n}.
\end{align*}
The existence of these bases is guaranteed by the fact that $W_h(E)$ is $\Sigma_E$-unisolvent and $\widetilde{W}_h(E)$ is $\widetilde{\Sigma}_E$-unisolvent. 

\noindent
 The above relationships allow a unique identification of the coefficients of the elements of the spaces $W_{h}(E)$ and $\widetilde{W}_{h}(E)$. In particular, any function $w_{h}\in W_{h}(E)$ and  $\widetilde{w}_{h}\in \widetilde{W}_{h}(E)$ can be uniquely represented with respect to the chosen bases as
\begin{align*}
&w_{h}(\mathbf{x})=\sum\limits_{i=1}^{n}w_{h,i}\lambda_{i}^{E}(\mathbf{x}) \quad \text{with} \quad w_{h,i}=\mathcal{L}_{i}\left(w_{h}\right) \quad \forall i=1,\ldots,n,\\
&\widetilde{w}_{h}(\mathbf{x})=\sum\limits_{j=1}^{\widetilde{n}}\widetilde{w}_{h,j}\widetilde{\lambda}_{j}^{E}(\mathbf{x}) \quad \text{with} \quad \widetilde{w}_{h,j}=\widetilde{\mathcal{L}}_{j}\left(\widetilde{w}_{h}\right) \quad \forall j=1,\ldots,\widetilde{n}.
\end{align*}
\noindent
These coefficients are referred to as the \textit{degrees of freedom} of the finite elements $\mathcal{S}$ and $\widetilde{\mathcal{S}}$, respectively. Due to the hypothesis made on the spaces $W_{h}(E)$ and $\widetilde{W}_{h}(E)$, we build the enriched space
\begin{equation}\label{local_enriched_space}
V_{h}(E):=W_{h}(E)\bigoplus\widetilde{W}_{h}(E) \quad \text{with} \quad m:=\text{dim}\left(V_{h}(E)\right)=n+\widetilde{n}.
\end{equation}
\noindent
Furthermore, requiring that the linear functionals of $\Sigma_{E}$ and $\widetilde{\Sigma}_{E}$ admit an extension over $V_{h}(E)$, we define the new set of linear functionals
\begin{equation*}
\Xi_E := \left\{\mathcal{L}_r, \widetilde{\mathcal{L}}_s:V_{h}(E)\rightarrow\mathbf{R} \, | \, r=1,\dots,n, \, s=1,\dots,\widetilde{n}\right\}.
\end{equation*}
\noindent
Although we can obtain a basis of the enriched space $V_h(E)$ as the union of the basis functions of $W_{h}(E)$ and $\widetilde{W}_{h}(E)$, it is more convenient to consider an alternative basis, made of functions $\varphi_{1}^{E},\varphi_{2}^{E},\ldots,\varphi_{m}^{E}$, satisfying the conditions
\begin{equation}\label{admissible_enrichment}
\mathcal{L}_{r}\left(\varphi_{i}^{E}\right)=\delta_{ir}, \quad \widetilde{\mathcal{L}}_{s}\left(\varphi_{i}^{E}\right)=0, \quad \mathcal{L}_{r}\left(\varphi_{j+n}^{E}\right)=0, \quad \widetilde{\mathcal{L}}_{s}\left(\varphi_{j+n}^{E}\right)=\delta_{js}, 
\end{equation}
$i,r=1,\dots,n$, and  $j,s=1,\dots,\widetilde{n}$.  The basis functions satisfying~\eqref{admissible_enrichment} are referred to as \textit{admissible enriched basis functions} relative to $V_{h}(E)$ and $\Xi_E$. If such a basis exists, the functions $\widetilde{\lambda}_1^E,\dots,\widetilde{\lambda}_{\widetilde{n}}^{E}$ are called \textit{admissible enrichment functions} relative to $V_h(E)$ and $\Xi_E$. Under these hypotheses, we can define the admissible enriched Ciarlet triplet 
$\mathcal{F}=\left(E,V_{h}(E),\Xi_E\right)$ and the global enriched space 
\begin{displaymath}
V_{h}:=\left\{v_{h}\in\text{H}_{0}^{1}(\Omega) \, | \, v_{h_{\mkern 1mu \vrule height 2ex\mkern2mu E}}\in V_{h}(E) \quad \forall E\in\mathcal{T}_{h}(\Omega)\right\}\subset V.
\end{displaymath}
\noindent
The derivation of the final linear system follows the general construction of the standard FEM. For this reason, one can easily write an enriched FEM code, starting from a standard one. Indeed, the biggest difference lies in the basis functions and, consequently, in the computation of local matrices needed to assemble the global matrix.
\section{General enrichment strategies of the linear lagrangian finite element} 
\noindent
In the following, we are interested in the enrichment of the linear lagrangian finite element. To this end, for each $E\in\mathcal{T}_{h}(\Omega)$, we consider the space of polynomials of degree less than or equal to $1$ on $E$, i.e.
\begin{displaymath}
W_{h}(E):=P^{1}(E) \quad \text{with} \quad n=\text{dim}\left(W_{h}(E)\right)=3.
\end{displaymath}
\noindent
A basis for this space can be constructed from the set of the \textit{barycentric coordinates} of the triangle $E$, i.e. 
\begin{eqnarray} \label{barcord}
\lambda_{1}^{E}({\mathbf{
x}})&:=&\frac{\left(\mathbf{x}-\mathbf{v}_{3}^{E}\right)\times\left(\mathbf{v}_{2}^{E}-\mathbf{v}_{3}^{E}\right)}{\left(\mathbf{v}_{1}^{E}-\mathbf{v}_{3}^{E}\right)\times\left(\mathbf{v}_{2}^{E}-\mathbf{v}_{3}^{E}\right)}, \notag
\\ 
\lambda_{2}^{E}({\mathbf{
x}})&:=&\frac{\left(\mathbf{x}-\mathbf{v}_{3}^{E}\right)\times\left(\mathbf{v}_{3}^{E}-\mathbf{v}_{1}^{E}\right)}{\left(\mathbf{v}_{1}^{E}-\mathbf{v}_{3}^{E}\right)\times\left(\mathbf{v}_{2}^{E}-\mathbf{v}_{3}^{E}\right)},\\ \notag \lambda_{3}^{E}({\mathbf{x}})&:=&\frac{\left(\mathbf{x}-\mathbf{v}_{1}^{E}\right)\times\left(\mathbf{v}_{1}^{E}-\mathbf{v}_{2}^{E}\right)}{\left(\mathbf{v}_{1}^{E}-\mathbf{v}_{3}^{E}\right)\times\left(\mathbf{v}_{2}^{E}-\mathbf{v}_{3}^{E}\right)}.
\end{eqnarray}
\noindent
These functions satisfy the partition of unity property, i.e.
\begin{equation}\label{partition_of_unity}
\sum\limits_{i=1}^{3}\lambda_{i}^{E}(\mathbf{x})=1, \quad \forall\mathbf{x}\in E,
\end{equation}
\noindent
and the Kronecker delta property, i.e.
\begin{equation}\label{deltaprop}
\lambda_{i}^{E}\left(\mathbf{v}^{E}_j\right)=\delta_{ij}.
\end{equation}
\noindent 
Moreover, since the following identity holds
\begin{equation}\label{prop21}
\lambda_{i}^{E}(t \mathbf{x} + (1-t) \mathbf{y}) = t \lambda_{i}^{E}(\mathbf{x}) + (1-t)\lambda_{i}^{E}(\mathbf{y}),  \quad \forall\mathbf{x}, \mathbf{y}\in E, \quad \forall t\in[0,1],
\end{equation}
\noindent
we deduce that the barycentric coordinate $\lambda_{i}^{E}$ vanishes on the edge $\mathbf{e}_{i}^{E}$, $i=1,2,3$. Introducing the set
\begin{displaymath}
\Sigma_{E}:=\left\{\mathcal{L}_{r}:W_{h}(E)\rightarrow\mathbf{R} \, | \, \mathcal{L}_{r}\left(w_{h}\right)=w_{h}\left(\mathbf{v}^E_{r}\right), \ r=1,2,3\right\},
\end{displaymath}
\noindent
the space $W_{h}(E)$ is $\Sigma_{E}$-unisolvent, since
\begin{equation*}
\mathcal{L}_{r}\left(\lambda_{i}^{E}\right):=\lambda_{i}^{E}\left(\mathbf{v}^E_{r}\right)=\delta_{ir}, \quad r=1,2,3. 
\end{equation*}
\noindent
As a consequence, any function $w_{h}\in W_{h}(E)$ is uniquely determined by its values at the vertices of $E$, which provide its degrees of freedom, i.e.
\begin{displaymath}
w_{h,r}=\mathcal{L}_{r}\left(w_{h}\right)=w_{h}\left(\mathbf{v}_{r}^{E}\right), \quad r=1,2,3.
\end{displaymath}
\noindent
Then, we can define the projection operator $\pi_{h}^{\mathrm{lin}}:C^{0}(E)\rightarrow W_{h}(E)$ such that
\begin{displaymath}
\pi_{h}^{\mathrm{lin}}[g](\mathbf{x})=\sum\limits_{r=1}^{3}\mathcal{L}_r(g) \lambda_{r}^{E}(\mathbf{x}), \quad g\in C^{0}(E)
\end{displaymath}
\noindent
and the associated approximation error 
\begin{equation}\label{eq:errorinterp}
\varepsilon_{h}^{\mathrm{lin}}[g](\mathbf{x}):=g(\mathbf{x})-\pi_{h}^{\mathrm{lin}}[g](\mathbf{x}), \quad g\in C^{0}(E).
\end{equation}
\noindent
Since we are interested in developing an enrichment strategy for the standard linear lagrangian finite element $\mathcal{S}=\left(E,W_{h}(E),\Sigma_{E}\right)$, we consider three linearly independent enrichment functions $\widetilde{\lambda}_{1}^{E}$, $\widetilde{\lambda}_{2}^{E}$,  $\widetilde{\lambda}_{3}^{E}$ and we define the space
\begin{displaymath}
\widetilde{W}_{h}(E):=\operatorname{span} \left\{\widetilde{\lambda}_{1}^{E}, \widetilde{\lambda}_{2}^{E}, \widetilde{\lambda}_{3}^{E} \right\} \quad \text{with} \quad \widetilde{n}=\text{dim}\left(\widetilde{W}_{h}(E)\right)=3.
\end{displaymath}
\noindent
If the functions $\widetilde{\lambda}_{1}^{E}$, $\widetilde{\lambda}_{2}^{E}$,  $\widetilde{\lambda}_{3}^{E}$ are chosen in such a way that $W_{h}(E)\cap\widetilde{W}_{h}(E)=\left\{\mathbf{0}\right\}$, we have all the ingredients to build the local enriched space
\begin{displaymath}
V_{h}(E):=W_{h}(E)\bigoplus\widetilde{W}_{h}(E) \quad \text{with} \quad m:=\text{dim}\left(V_{h}\right)=6.
\end{displaymath}
\noindent
Then, we introduce the enriched linear functionals $\widetilde{\mathcal{L}}_{j}:\widetilde{W}_{h}(E)\rightarrow\mathbf{R}$ such that
\begin{equation}\label{enrichment_linear_functionals}
\widetilde{\mathcal{L}}_j(w):=\frac{1}{ \left\lvert \mathbf{e}_j^{E}\right\rvert}\int_{ \mathbf{e}_j^{E}} w(\mathbf{x}) d \mathbf{x}, \quad j=1,2,3.
\end{equation}
\noindent 
Assuming that all the standard functionals $\mathcal{L}_{1},\mathcal{L}_{2},\mathcal{L}_{3}$ and the enriched ones $\widetilde{\mathcal{L}}_{1},\widetilde{\mathcal{L}}_{2},\widetilde{\mathcal{L}}_{3}$ can be extended to the space $V_{h}(E)$, we collect them into the set
\begin{equation}\label{newXiE}
\Xi_{E}:=\left\{\mathcal{L}_{i}, \widetilde{\mathcal{L}}_{j}:V_{h}(E)\rightarrow\mathbf{R} \, | \, i,j=1,2,3 \right\}. 
\end{equation}
\noindent
We now recall a characterization result, whose proof can be found in~\cite{DellAccio2023AGC}, that establishes the conditions on the enrichment functions and the enriched linear functionals under which the enriched Ciarlet triple $\mathcal{F}:=\left(E,V_{h}(E),\Xi_{E}\right)$ is a finite element. To this aim, we consider the matrix
\begin{equation}\label{dnn}
\mathbb{G}= \begin{pmatrix}
\widetilde{\mathcal{G}}_1\left(\widetilde{\lambda}_1^{E}\right) & \widetilde{\mathcal{G}}_1\left(\widetilde{\lambda}_2^{E}\right)&\widetilde{\mathcal{G}}_1\left(\widetilde{\lambda}_3^{E}\right) \\
\widetilde{\mathcal{G}}_2\left(\widetilde{\lambda}_1^{E}\right) & \widetilde{\mathcal{G}}_2\left(\widetilde{\lambda}_2^{E}\right)&\widetilde{\mathcal{G}}_2\left(\widetilde{\lambda}_3^{E}\right)\\
\widetilde{\mathcal{G}}_3\left(\widetilde{\lambda}_1^{E}\right) & \widetilde{\mathcal{G}}_3\left(\widetilde{\lambda}_2^{E}\right)&\widetilde{\mathcal{G}}_3\left(\widetilde{\lambda}_3^{E}\right)
\end{pmatrix},
\end{equation}
\noindent
where the new linear functionals $\widetilde{\mathcal{G}}_{j}:V_{h}(E)\rightarrow\mathbf{R}$ are defined by
\begin{equation}\label{errf}
 \widetilde{\mathcal{G}}_{j}\left(v_{h}\right):= \widetilde{\mathcal{L}}_j\left(v_{h}\right)-\sum_{i=1}^{3} \widetilde{\mathcal{L}}_j\left(\lambda_i^{E}\right)\mathcal{L}_i\left(v_{h}\right), 
\quad j=1,2,3.
\end{equation}
\begin{theorem}\label{th3}
Let $\widetilde{\lambda}_{1}^{E},\widetilde{\lambda}_{2}^{E},\widetilde{\lambda}_{3}^{E}$ be enrichment functions. Then the following conditions are equivalent:
\begin{itemize}
    \item [1)] The enriched Ciarlet triple $\mathcal{F}=\left(E,V_{h}(E),\Xi_{E}\right)$ is a finite element;
    \item[2)] The matrix $\mathbb{G}$ is such that $\operatorname{Ker}(\mathbb{G})=\{0\}$.
\end{itemize}
\end{theorem}
\begin{remark}\label{remark_impo}
We point out that, if the enrichment functions $\widetilde{\lambda}_1^{E},\widetilde{\lambda}_2^{E},\widetilde{\lambda}_3^{E}$ satisfy the vanishing conditions at the vertices of the triangle $E$, i.e. 
\begin{equation}\label{vancond}
\mathcal{L}_j\left(\widetilde{\lambda}_i^{E}\right)=0, \quad i,j=1,2,3,
\end{equation}
\noindent 
the matrix $\mathbb{G}$, defined in~\eqref{dnn}, becomes
\begin{equation}\label{dnnn}
\mathbb{G}= \begin{pmatrix}
\widetilde{\mathcal{L}}_1\left(\widetilde{\lambda}_1^{E}\right) & \widetilde{\mathcal{L}}_1\left(\widetilde{\lambda}_2^{E}\right)&\widetilde{\mathcal{L}}_1\left(\widetilde{\lambda}_3^{E}\right) \\
\widetilde{\mathcal{L}}_2\left(\widetilde{\lambda}_1^{E}\right) & \widetilde{\mathcal{L}}_2\left(\widetilde{\lambda}_2^{E}\right)&\widetilde{\mathcal{L}}_2\left(\widetilde{\lambda}_3^{E}\right) \\
\widetilde{\mathcal{L}}_3\left(\widetilde{\lambda}_1^{E}\right) & \widetilde{\mathcal{L}}_3\left(\widetilde{\lambda}_2^{E}\right)&\widetilde{\mathcal{L}}_3\left(\widetilde{\lambda}_3^{E}\right)
\end{pmatrix}.
\end{equation}
\noindent
The properties of the enrichment functions satisfying~\eqref{vancond} are discussed in~\cite{DellAccio:2022:AUE}. 
\end{remark}
\noindent
Throughout the paper, we assume that 
$\mathbb{G}$ is invertible and we denote its inverse by
\begin{equation}\label{mfinvvv}
\mathbb{G}^{-1}=\left({\mathbf{g}}_1^{E} \ {\mathbf{g}}_2^{E} \ {\mathbf{g}}_3^{E}\right) \quad \text{with} \quad \mathbf{g}_i^{E}=\left(g_{i,1}^{E},g_{i,2}^{E},g_{i,3}^{E}\right)^{\top}\in\mathbf{R}^3, \quad i=1,2,3.
\end{equation}
\noindent
This matrix is the starting point to derive an analytical expression of the admissible enriched basis functions $\varphi_{1}^E,\varphi_{2}^E,\varphi_{3}^E,\varphi_{4}^E,\varphi_{5}^E,\varphi_{6}^E$ relative to the space $V_{h}(E)$ and the set $\Xi_{E}$, as summarized in the following result (see~\cite{DellAccio2023AGC}).
\begin{theorem}\label{th4}
\noindent
The admissible enriched basis functions relative to the space $V_{h}(E)$ and the set $\Xi_{E}$ have the following expressions
\begin{equation}\label{ggb1g}
    \varphi^{E}_i= \lambda_i^{E}-\frac{1}{2}\sum\limits_{\substack{ %
j=1  \\ j\neq i}}^3\varphi_{j+3}^{E} \quad \text{and} \quad \varphi_{i+3}^{E}= \left\langle \varepsilon_{h}\left[\boldsymbol{\widetilde{\lambda}}^{E}\right],{\mathbf{g} }_i^{E}\right \rangle, \quad i=1,2,3, 
\end{equation}
\noindent
where $\widetilde{{\boldsymbol{\lambda}}}^{E}:=\left(\widetilde{\lambda}_1^{E}, \widetilde{\lambda}_2^{E}, \widetilde{\lambda}_3^{E}\right)^{\top}$ and $\left\langle{\cdot}, {\cdot}\right\rangle$ denotes the standard scalar product in the Euclidean space $\mathbf{R}^2$. 
\end{theorem}
\begin{remark}
We highlight that, if we consider enrichment functions which satisfy the vanishing conditions~\eqref{vancond}, then the basis functions~\eqref{ggb1g} become 
\begin{equation}\label{ggwb1g}
\varphi_i^{E}= \lambda_i^{E}-\frac{1}{2}\sum\limits_{\substack{ %
j=1  \\ j\neq i}}^3\varphi_{j+3}^{E} \quad \text{and} \quad \varphi_{i+3}^{E}= \left\langle {\boldsymbol{\widetilde{\lambda}}}^{E},{\mathbf{g}}_i^{E}\right \rangle, \quad  i=1,2,3.
\end{equation}
\end{remark}
\subsection{Weighted and unweighted enrichment functions}
\noindent
In order to provide several examples of admissible enrichment functions $\widetilde{\lambda}_{1}^{E},\widetilde{\lambda}_{2}^{E},\widetilde{\lambda}_{3}^{E}$ relative to $V_h(E)$ and $\Xi_E$, we recall two important results, introduced for the first time in~\cite{DellAccio2023AGC}. 
\begin{theorem} \label{theoremprod}
Let $n\in \mathbf{N}$ and let $f_1^{E},\dots,f_n^{E}\in C^0([0,1])$ be convex, increasing (or decreasing) and nonnegative functions different from zero. If we assume that at least one of these functions is strictly convex, the enrichment functions
\begin{equation} \label{enrfun1type} 
\widetilde{\lambda}_i^{E}(\mathbf{x})=\prod_{k=1}^{n}f^{E}_{k}\left({\lambda}_i^{E}(\mathbf{x})\right), \quad i=1,2,3,
\end{equation}
are admissible.
\end{theorem}
\begin{example} \label{rem123}
Due to the above theorem, we can enrich the linear lagrangian finite element to the finite element $\mathcal{F}$ by using the following sets of admissible enrichment functions
    \begin{itemize}
        \item $\mathcal{E}_1=\left\{ \widetilde{\lambda}_i^{E}=\sin\left(\frac{\pi}{2}\left(\lambda_i^{E}+2\right)\right)+2\, |\,  i=1,2,3\right\},$
       \item $\mathcal{E}_2=\left\{ \widetilde{\lambda}_i^{E}=\frac{1}{1+\lambda_i^{E}}\, | \, i=1,2,3\right\},$
       \item $\mathcal{E}_3=\left\{ \widetilde{\lambda}_i^{E}=e^{\lambda_i^{E}} \, | \, i=1,2,3\right\}, $
       \item $\mathcal{E}_4=\left\{ \widetilde{\lambda}_i^{E}= \left(\lambda_i^{E}\right)^{\alpha} \, | \, \alpha>1, \,  i=1,2,3\right\}, $
        \item $\mathcal{E}_5=\left\{ \widetilde{\lambda}_i^{E}=\left(\lambda_i^{E}\right)^{\alpha}e^{\lambda_i^{E}}\, | \, \alpha>1, \, i=1,2,3\right\}. $
    \end{itemize}
\end{example}
\begin{theorem}\label{thlast}
Let $f_1^{E},f_2^{E},f_{3}^{E}\in C^0([0,1])$ be continuous functions such that $f_i^{E}(0)\neq 0$, $i=1,2,3,$ and let $\alpha_1,\alpha_2,\alpha_{3}>1$. Then, the enrichment functions
\begin{equation} \label{eq:secondclassofef}
\widetilde{\lambda}_i^{E}(\mathbf{x})=f_i^{E}\left({\lambda}_i^{E}({\mathbf{x}})\right) \prod_{\substack{k=1\\k\neq i}}^{3}\left({\lambda}^{E}_k({\mathbf{x}})\right)^{\alpha_k-1}, \quad i=1,2,3,
\end{equation}
\noindent
are admissible.
\end{theorem}
\begin{example}\label{enres1}   
Due to the above theorem, we can  enrich the linear lagrangian finite element to the finite element ${\mathcal{F}}$ by using the following sets of admissible enrichment functions
\begin{itemize}
\item $\mathcal{E}_6=\left\{ \widetilde{\lambda}_i^{E}=\sin\left(\frac{\pi}{2i+2}\left(\lambda^{E}_i+1\right)\right)\prod\limits_{\substack{k=1\\k\neq i}}^{3} \left(\lambda^{E}_k\right)^{\alpha_k-1}\, | \, \alpha_k>1, \, k,i=1,2,3 \right\}, $
\item $\mathcal{E}_7=\left\{ \widetilde{\lambda}_i^{E}=\frac{\imath+1}{1+\lambda_i^{E}}\prod\limits_{\substack{k=1\\k\neq i}}^{3} \left(\lambda^{E}_k\right)^{\alpha_k-1}\, | \, \alpha_k> 1, \, k,i=1,2,3\right\}, $
\item $\mathcal{E}_8=\left\{ \widetilde{\lambda}_i^{E}=e^{\imath\lambda_i^{E}}\prod\limits_{\substack{k=1\\k\neq i}}^{3} \left(\lambda^{E}_k\right)^{\alpha_k-1}\, | \, \alpha_k>1, \, k,i=1,2,3\right\}, $
\item $\mathcal{E}_9=\left\{ \widetilde{\lambda}_i^{E}= \log\left(\imath\lambda_i^{E}+2\right)\prod\limits_{\substack{k=1\\k\neq i}}^{3} \left(\lambda^{E}_k\right)^{\alpha_k-1}\, | \, \alpha_k>1, \, k,i=1,2,3\right\}. $
\end{itemize}
\end{example}
\noindent
At this stage, we aim to introduce two new three-parameter families of weighted enrichment functions, which will be extensively used in the numerical examples illustrated in Section~\ref{sec4}. To simplify the notations, from now on we use the standard cyclic convention
\begin{equation*}
    \mathbf{v}_4^{E} := \mathbf{v}_1^{E}, \quad \mathbf{v}_5^{E} := \mathbf{v}_2^{E}, \quad \lambda_4^{E} := \lambda_1^{E}, \quad \lambda_5^{E} := \lambda_2^{E}.
\end{equation*}
\noindent
Taking inspiration from~\cite{nudo5082357general}, we consider the weight function
\begin{equation}\label{defofh}
\omega^{E}_{ \mu ,\alpha,\beta}\left(\mathbf{x}\right)= \sum_{j=1}^{3}\left(1-\lambda_{j}^{E}\left(\mathbf{x}\right)\right)^\mu\left(\lambda^{E}_{j+1}\left(\mathbf{x}\right)\right)^{\alpha}\left(\lambda^{E}_{j+2}\left(\mathbf{x}\right)\right)^{\beta}, \quad \mu\ge0, \quad \alpha,\beta>-1.
\end{equation}
\noindent
A straightforward computation shows that the restriction of $\omega^{E}_{ \mu ,\alpha,\beta}\left(\mathbf{x}\right)$ to any edge of $E$ coincides with the classical Jacobi weight function on $[0,1]$, i.e.
\begin{equation}\label{weig}
\omega_{\mu,\alpha,\beta}^{E} \left(t\mathbf{v}_{j+1}^{E} + (1-t)\mathbf{v}_{j+2}^{E}\right)=\omega_{\alpha,\beta}(t):=t^{\alpha}\left(1-t\right)^{\beta}, \quad t\in[0,1], \quad j=1,2,3.
 \end{equation}
\noindent
The weight function $\omega^{E}_{ \mu ,\alpha,\beta}$ is used in the following two theorems, to introduce three-parameter families of admissible weighted enrichment functions.
\begin{theorem}
Let $I_{1}, I_{2}$, $I_{3}$ be finite sets of integers and let $f_{\ell}^{E}:[0,1]\rightarrow\mathbf{R}$,  $\ell\in I_{1}\cup I_{2}\cup I_{3}$, be integrable functions. We assume that these functions are strictly positive almost everywhere on $[0,1]$, i.e. except on a set of measure zero, and that, for each $k=1,2,3$, there exist two indices $r_k,z_k\in I_k$ such that 
\begin{equation*}
f_{r_k}^E(0)=f_{z_k}^E(1)=0.    
\end{equation*}
Then,  for any choice of the parameters $\mu\ge0$ and $\alpha,\beta>-1$, the enrichment functions
\begin{equation}\label{enrfuns1}
\widetilde{\lambda}^{E}_k(\mathbf{x})=\omega^{E}_{ \mu ,\alpha,\beta}\left(\mathbf{x}\right)\prod_{\ell\in I_k} f_{\ell}^{E}\left(\lambda_k^{E}(\mathbf{x})\right), \quad k=1,2,3
\end{equation}
are admissible.
\end{theorem}
\begin{proof}
\noindent
To obtain the thesis, we first show that the enrichment functions $\widetilde{\lambda}^{E}_{1},\widetilde{\lambda}^{E}_{2}$, $\widetilde{\lambda}^{E}_{3}$, vanish at the vertices $\mathbf{v}^E_{1},\mathbf{v}^E_{2},\mathbf{v}^E_{3}$ of the triangle $E$, i.e.
\begin{equation*}
\widetilde{\lambda}_{k}^{E}\left(\mathbf{v}^E_j\right)=0, \quad j,k=1,2,3.
\end{equation*}
\noindent
Due to the Kronecker delta property~\eqref{deltaprop}, we have
\begin{equation*}
\widetilde{\lambda}_{k}^{E}\left(\mathbf{v}^E_j\right)= \omega^{E}_{ \mu ,\alpha,\beta}\left(\mathbf{v}^E_j\right)\prod_{\ell\in I_{k}} f_{\ell}^{E}\left(\lambda_{k}^E\left(\mathbf{v}^E_j\right)\right)= \omega^{E}_{ \mu ,\alpha,\beta}\left(\mathbf{v}^E_j\right)\prod_{\ell\in I_{k}} f_{\ell}^{E}\left(\delta_{kj}\right). 
\end{equation*}
\noindent
Now, if $j\neq k$, it holds
\begin{equation*}
\widetilde{\lambda}_{k}^{E}\left(\mathbf{v}^E_j\right)=  \omega^{E}_{ \mu ,\alpha,\beta}\left(\mathbf{v}^E_j\right)\prod_{\ell\in I_{k}} f_{\ell}^{E}\left(0\right)=0 
\end{equation*}
\noindent
because the hypotheses of the theorem ensure the existence of an index $r_{k}\in I_{k}$ such that $f_{r_{k}}^E\left(0\right)=0$. On the other hand, when $j=k$, we obtain
\begin{equation*}
\widetilde{\lambda}_{k}^{E}\left(\mathbf{v}^E_j\right)=  \omega^{E}_{ \mu ,\alpha,\beta}\left(\mathbf{v}^E_j\right)\prod_{\ell\in I_{k}} f_{\ell}^{E}\left(1\right)=0 
\end{equation*}
\noindent
since there exists an index $z_{k}\in I_{k}$ such that $f_{z_{k}}^E\left(1\right)=0$. As a consequence, the matrix defined in~\eqref{dnn} can be written as
\begin{displaymath}
\mathbb{G}= \begin{pmatrix}
\widetilde{\mathcal{L}}_1\left(\widetilde{\lambda}_1^{E}\right) & \widetilde{\mathcal{L}}_1\left(\widetilde{\lambda}_2^{E}\right)&\widetilde{\mathcal{L}}_1\left(\widetilde{\lambda}_3^{E}\right) \\
\widetilde{\mathcal{L}}_2\left(\widetilde{\lambda}_1^{E}\right) & \widetilde{\mathcal{L}}_2\left(\widetilde{\lambda}_2^{E}\right)&\widetilde{\mathcal{L}}_2\left(\widetilde{\lambda}_3^{E}\right) \\
\widetilde{\mathcal{L}}_3\left(\widetilde{\lambda}_1^{E}\right) & \widetilde{\mathcal{L}}_3\left(\widetilde{\lambda}_2^{E}\right)&\widetilde{\mathcal{L}}_3\left(\widetilde{\lambda}_3^{E}\right)
\end{pmatrix}.
\end{displaymath}
\noindent
Now, we prove that only the vector $\mathbf{0}\in\mathbf{R}^{3}$ belongs to the kernel of the matrix $\mathbb{G}$. To this aim, we consider the following parametrization of the edge $\mathbf{e}^{E}_1$:
\begin{eqnarray*}
t&\rightarrow& t {\mathbf{v}}^{E}_{2}+(1-t){\mathbf{v}}^{E}_{3}, \quad t\in[0,1].
\end{eqnarray*}
\noindent
Using the properties~\eqref{deltaprop} and~\eqref{prop21} of the barycentric coordinates, it holds
\begin{eqnarray*}
\widetilde{\mathcal{L}}_1\left(\widetilde{\lambda}_1^{E}\right)&=&\frac{1}{\left\lvert \mathbf{e}_1^{E}\right\rvert}\int_{ \mathbf{e}_1^{E}}\widetilde{\lambda}_1^{E}(\mathbf{x})d \mathbf{x}=\int_{0}^1 \omega_{\alpha,\beta}(t) \prod_{\ell\in I_1} f_{\ell}^{E}\left(\lambda_1^{E}\left(t {\mathbf{v}}^{E}_{2}+(1-t){\mathbf{v}}^{E}_{3}\right)\right)d t=\\
&=&\int_{0}^1 \omega_{\alpha,\beta}(t) \prod_{\ell\in I_1} f_{\ell}^{E}\left(t\lambda_1^{E}\left( {\mathbf{v}}^{E}_{2}\right)+(1-t)\lambda_1^{E}\left({\mathbf{v}}^{E}_{3}\right)\right)d t=\\
&=&\int_{0}^1 \omega_{\alpha,\beta}(t) \prod_{\ell\in I_1} f_{\ell}^{E}\left(0\right)d t=0
\end{eqnarray*}
\noindent
since, by assumption, there exists an index $r_{1}\in I_{1}$ such that $f_{r_{1}}^E\left(0\right)=0$. Analogously, considering the following parametrization of the edge $\mathbf{e}^{E}_2$:
\begin{eqnarray*}
t&\rightarrow& t {\mathbf{v}}^{E}_{3}+(1-t){\mathbf{v}}^{E}_{1}, \quad t\in[0,1]
\end{eqnarray*}
\noindent
and recalling that the functions $f_{\ell}^{E}$ are strictly positive almost everywhere for $\ell\in I_1$, we have
\begin{eqnarray*}
\widetilde{\mathcal{L}}_2\left(\widetilde{\lambda}_1^{E}\right)&=&\frac{1}{\left\lvert \mathbf{e}_2^{E}\right\rvert}\int_{ \mathbf{e}_2^{E}}\widetilde{\lambda}_1^{E}(\mathbf{x})d \mathbf{x}=\int_{0}^1 \omega_{\alpha,\beta}(t) \prod_{\ell\in I_1} f_{\ell}^{E}\left(\lambda_1^{E}\left(t {\mathbf{v}}^{E}_{3}+(1-t){\mathbf{v}}^{E}_{1}\right)\right)d t=\\
&=&\int_{0}^1 \omega_{\alpha,\beta}(t) \prod_{\ell\in I_1} f_{\ell}^{E}\left(t\lambda_1^{E}\left( {\mathbf{v}}^{E}_{3}\right)+(1-t)\lambda_1^{E}\left({\mathbf{v}}^{E}_{1}\right)\right)d t=\\
&=&\int_{0}^1 \omega_{\alpha,\beta}(t) \prod_{\ell\in I_1} f_{\ell}^{E}\left(1-t\right)d t>0.
\end{eqnarray*}
\noindent
A similar argument shows that $\widetilde{\mathcal{L}}_3\left(\widetilde{\lambda}_1^{E}\right)>0$ and, more generally, that
\begin{equation*}
\widetilde{\mathcal{L}}_i\left(\widetilde{\lambda}_j^{E}\right)=\left(1-\delta_{ij}\right)\mathbb{G}_{ij} \quad \text{with} \quad \mathbb{G}_{ij}>0, \quad i,j=1,2,3. 
\end{equation*}
\noindent
Thus, the matrix $\mathbb{G}$ is such that
\begin{equation}
   \mathbb{G}= \begin{bmatrix}
0 & \mathbb{G}_{12}& \mathbb{G}_{13} \\
\mathbb{G}_{21} & 0 & \mathbb{G}_{23} \\
\mathbb{G}_{31} & \mathbb{G}_{32} & 0
\end{bmatrix} \quad \text{and} \quad \operatorname{det}\left(\mathcal{\mathbb{G}}\right)=
\mathbb{G}_{12}\mathbb{G}_{23}\mathbb{G}_{31}+\mathbb{G}_{13}\mathbb{G}_{21}\mathbb{G}_{32}>0.
\end{equation}
\noindent
We can conclude that $\operatorname{Ker}(\mathcal{\mathbb{G}})=\{{0}\}$ and the thesis follows from Theorem~\ref{th3}.
\end{proof}
\begin{theorem}\label{thmnew} 
Let $f_{\ell}^{E}:[0,1]\rightarrow\mathbf{R}$, $\ell=1,\dots,6$,  be integrable functions with a constant sign. We assume that  $f_{\ell}^{E}\neq 0$ almost everywhere and that
\begin{equation}\label{hypimp}
f_\ell^{E}(0)=0, \quad \ell=1,\dots,6.
\end{equation}
Then, for any choice of the parameters $\mu\ge0$ and $\alpha,\beta>-1$, the enrichment functions
\begin{equation}\label{enrfuns13}
\widetilde{\lambda}^{E}_k(\mathbf{x})=\omega^{E}_{ \mu ,\alpha,\beta}\left(\mathbf{x}\right)f_{2k-1}^{E}\left(\lambda_{k+1}^{E}(\mathbf{x})\right)f_{2k}^{E}\left(\lambda_{k+2}^{E}(\mathbf{x})\right), \quad k=1,2,3
\end{equation}
are admissible.    
\end{theorem}
\begin{proof}
To prove this theorem, it suffices to show that $\mathbf{z}=\mathbf{0}\in\mathbf{R}^{3}$ is the only vector such that $\mathbb{G}\mathbf{z}=\mathbf{0}$, where the matrix $\mathbb{G}$ is defined in~\eqref{dnn}. Due to the Kronecker delta property of the barycentric coordinates, relationship~\eqref{hypimp} allows us to conclude that the enrichment functions $\widetilde{\lambda}^{E}_{1},\widetilde{\lambda}^{E}_{2}$, $\widetilde{\lambda}^{E}_{3}$ satisfy the condition~\eqref{vancond}, i.e.
\begin{displaymath}
\widetilde{\lambda}^{E}_k\left(\mathbf{v}^E_{j}\right)=\omega^{E}_{ \mu ,\alpha,\beta}\left(\mathbf{v}^E_{j}\right)f_{2k-1}^{E}\left(\lambda_{k+1}^{E}\left(\mathbf{v}^E_{j}\right)\right)f_{2k}^{E}\left(\lambda_{k+2}^{E}\left(\mathbf{v}^E_{j}\right)\right)=0, \quad j,k=1,2,3,
\end{displaymath}
\noindent
where $\mathbf{v}^E_{1},\mathbf{v}^E_{2},\mathbf{v}^E_{3}$ are the vertices of the triangle $E$. As a consequence (see remark \ref{remark_impo}), the matrix $\mathbb{G}$ is given by
\begin{equation}
\mathbb{G}= \begin{pmatrix}
\widetilde{\mathcal{L}}_1\left(\widetilde{\lambda}_1^{E}\right) & \widetilde{\mathcal{L}}_1\left(\widetilde{\lambda}_2^{E}\right)&\widetilde{\mathcal{L}}_1\left(\widetilde{\lambda}_3^{E}\right) \\
\widetilde{\mathcal{L}}_2\left(\widetilde{\lambda}_1^{E}\right) & \widetilde{\mathcal{L}}_2\left(\widetilde{\lambda}_2^{E}\right)&\widetilde{\mathcal{L}}_2\left(\widetilde{\lambda}_3^{E}\right) \\
\widetilde{\mathcal{L}}_3\left(\widetilde{\lambda}_1^{E}\right) & \widetilde{\mathcal{L}}_3\left(\widetilde{\lambda}_2^{E}\right)&\widetilde{\mathcal{L}}_3\left(\widetilde{\lambda}_3^{E}\right)
\end{pmatrix}.
\end{equation}
\noindent
To compute its entries, we start by considering the following parametrization of the edge $\mathbf{e}^{E}_1$:
\begin{eqnarray*}
t&\rightarrow& t {\mathbf{v}}^{E}_{2}+(1-t){\mathbf{v}}^{E}_{3}, \quad t\in[0,1]
\end{eqnarray*}
\noindent
and, taking into account the relationships~\eqref{deltaprop} and~\eqref{prop21}, we get
\begin{eqnarray*}
\widetilde{\mathcal{L}}_1\left(\widetilde{\lambda}_1^{E}\right)&=&\frac{1}{ \left\lvert \mathbf{e}_1^{E}\right\rvert}\int_{ \mathbf{e}_1^{E}} h^{E}_{ \mu ,\alpha,\beta}\left(\mathbf{x}\right)f_{1}^{E}\left(\lambda_{2}^{E}(\mathbf{x})\right)f_{2}^{E}\left(\lambda_{3}^{E}(\mathbf{x})\right) d \mathbf{x}=\\
&=&\int_{0}^1 \omega_{\alpha,\beta}(t)f^{E}_{1}\left(\lambda_2^{E}\left(t {\mathbf{v}}^{E}_{2}+(1-t){\mathbf{v}}^{E}_{3}\right)\right)f^{E}_2\left(\lambda_3^{E}\left(t {\mathbf{v}}^{E}_{2}+(1-t){\mathbf{v}}^{E}_{3}\right)\right)d t=\\
&=&\int_{0}^1 \omega_{\alpha,\beta}(t)f^{E}_{1}\left(t\lambda_2^{E}\left({\mathbf{v}}^{E}_{2}\right)+(1-t) \lambda_2^{E}\left({\mathbf{v}}^{E}_{3}\right)\right)f^{E}_2\left(t\lambda_3^{E}\left( {\mathbf{v}}^{E}_{2}\right)+(1-t)\lambda_3^{E}\left({\mathbf{v}}^{E}_{3}\right)\right)d t=\\
&=& \int_{0}^1  \omega_{\alpha,\beta}(t)f_{1}^{E}(t)f_2^{E}(1-t)d t\neq 0,\\
\end{eqnarray*}
\noindent
because, by hypothesis, the functions $f_1^{E}$, $f_2^{E}$ have a constant sign over the interval $[0,1]$ and are nonzero almost everywhere. On the other hand, we have
\begin{equation*}
\widetilde{\mathcal{L}}_1\left(\widetilde{\lambda}^{E}_2\right)=\frac{1}{ \left\lvert \mathbf{e}_1^{E}\right\rvert}\int_{ \mathbf{e}_1^{E}}\widetilde{\lambda}_2^{E}(\mathbf{x}) d \mathbf{x}=\frac{1}{ \left\lvert \mathbf{e}_1^{E}\right\rvert}\int_{ \mathbf{e}_1^{E}} \omega_{\mu,\alpha,\beta}^{E}(\mathbf{x})f_{3}^{E}\left(\lambda^{E}_{3}(\mathbf{x})\right)f_{4}^{E}\left(\lambda^{E}_{1}(\mathbf{x})\right) d \mathbf{x}=0,
\end{equation*}
and 
\begin{equation*}
\widetilde{\mathcal{L}}_1\left(\widetilde{\lambda}_3^{E}\right)=\frac{1}{ \left\lvert \mathbf{e}_1^{E}\right\rvert}\int_{ \mathbf{e}_1^{E}}\widetilde{\lambda}_3^{E}(\mathbf{x}) d \mathbf{x}=\frac{1}{ \left\lvert \mathbf{e}_1^{E}\right\rvert}\int_{ \mathbf{e}_1^{E}} \omega_{\mu,\alpha,\beta}^{E}(\mathbf{x})f_{5}^{E}\left(\lambda_{1}^{E}(\mathbf{x})\right)f_{6}^{E}\left(\lambda_{2}^{E}(\mathbf{x})\right) d \mathbf{x}=0,
\end{equation*}
since the barycentric coordinate $\lambda_{1}^{E}$ vanishes on the edge $\mathbf{e}_{1}^{E}$. In exactly the same way, we obtain
\begin{equation*}
\widetilde{\mathcal{L}}_2\left(\widetilde{\lambda}^{E}_1\right)=\widetilde{\mathcal{L}}_2\left(\widetilde{\lambda}^{E}_3\right)=0, \quad \text{and} \quad \widetilde{\mathcal{L}}_2\left(\widetilde{\lambda}_2^{E}\right)\neq 0,
\end{equation*}
\noindent
and 
\begin{equation*}
\widetilde{\mathcal{L}}_3\left(\widetilde{\lambda}^{E}_1\right)=\widetilde{\mathcal{L}}_3\left(\widetilde{\lambda}^{E}_2\right)=0, \quad \text{and} \quad \widetilde{\mathcal{L}}_3\left(\widetilde{\lambda}^{E}_3\right)\neq 0.
\end{equation*}
Thus, $\mathbb{G}$ is a diagonal matrix, i.e.
\begin{equation}\label{matN}
   \mathbb{G}= \begin{bmatrix}
\mathbb{G}_{11} & 0&0 \\
0 & \mathbb{G}_{22}&0 \\
0 & 0 & \mathbb{G}_{33}
\end{bmatrix} \quad \text{with} \quad \mathbb{G}_{ii}\neq0, \quad i=1,2,3.
\end{equation}
\noindent
Since each diagonal entry is nonzero, we can conclude that $\mathbf{z}=\mathbf{0}\in\mathbf{R}^{3}$ is the only vector such that $\mathbb{G}\mathbf{z}=\mathbf{0}$, i.e. $\operatorname{Ker}(\mathcal{\mathbb{G}})=\{0\}$, and the thesis follows from Theorem~\ref{th3}.
\end{proof}

\begin{example}
Due to Theorem~\ref{thmnew}, we can  enrich the linear lagrangian finite element to the finite element ${\mathcal{F}}$ by using the following sets of admissible enrichment functions
\begin{itemize}
\item $\mathcal{E}_{10}=\left\{ \widetilde{\lambda}^{E}_i=\omega^{E}_{ \mu ,\alpha,\beta}\sin\left(\lambda^{E}_{i+1}\right)\sin\left(\lambda^{E}_{i+2}\right) \, |  \, \mu\ge0, \, \alpha,\beta>-1, \,  i=1,2,3 \right\}, $
\item $\mathcal{E}_{11}=\left\{ \widetilde{\lambda}^{E}_i=\omega^{E}_{ \mu ,\alpha,\beta}\left(e^{\lambda_{i+1}^{E}}-1\right)\left(e^{\lambda_{i+2}^{E}}-1\right) \, |  \, \mu\ge0, \, \alpha,\beta>-1, \,  i=1,2,3 \right\},$
\item $\mathcal{E}_{12}=\left\{ \widetilde{\lambda}^{E}_i=\omega^{E}_{ \mu ,\alpha,\beta}\left(e^{\lambda_{i+1}^{E}}-1\right)\sin\left(\lambda^{E}_{i+2}\right) \, |  \, \mu\ge0, \, \alpha,\beta>-1, \,  i=1,2,3 \right\},$
\item $\mathcal{E}_{13}=\left\{ \widetilde{\lambda}^{E}_i=\omega^{E}_{ \mu ,\alpha,\beta}\sin\left(\lambda_{i+1}^{E}\right)\left(\cos\left(\lambda_{i+2}^{E}\right)-1\right) \, |  \, \mu\ge0, \, \alpha,\beta>-1, \,  i=1,2,3 \right\}, $
\item $\mathcal{E}_{14}=\left\{ \widetilde{\lambda}^{E}_i=\omega^{E}_{ \mu ,\alpha,\beta}\log\left(\lambda_{i+1}^{E}+1\right)\lambda_{i+2}^{E} \, |  \, \mu\ge0, \, \alpha,\beta>-1, \,  i=1,2,3 \right\},$
\item $\mathcal{E}_{15}=\left\{ \widetilde{\lambda}^{E}_i=\omega^{E}_{ \mu ,\alpha,\beta}\left(\lambda_{i+1}^{E}\right)^{\alpha_1}\left(\lambda_{i+2}^{E}\right)^{\beta_1} \, |  \, \mu\ge0, \, \alpha,\beta>-1, \, \alpha_1,\beta_1\ge0, \,  i=1,2,3 \right\}. $
    \end{itemize}
\end{example}
\begin{remark}
We point out that, if we choose $\mu=\alpha=\beta=0$,  the weight function~\eqref{defofh} reduces to $h^{E}_{0,0,0}=1$.  In our numerical experiments, we use this special case to assess the unweighted enrichment strategy.
\end{remark}
\begin{remark}\label{remimp}
\noindent
Since the barycentric coordinate $\lambda_{i}^{E}$ vanishes on the edge $\mathbf{e}_{i}^{E}$, it results
\begin{equation*}
\widetilde{\mathcal{L}}_i\left(\lambda_i^{E}\right)=0, \quad i=1,2,3.    
\end{equation*}
On the other hand, if $i\neq j$, it holds
\begin{equation*}
\widetilde{\mathcal{\mathcal{L}}}_i\left(\lambda_j^{E}\right)=\frac{1}{ \left\lvert \mathbf{e}_i^{E}\right\rvert}\int_{ \mathbf{e}_i^{E}}\lambda^{E}_j(\mathbf{x}) d \mathbf{x}=\int_{0}^1 \lambda_j^{E}\left(t\mathbf{v}^{E}_{i+1}+(1-t)\mathbf{v}^{E}_{i+2}\right)  d t=\frac{1}{2}.
\end{equation*}
\noindent
Furthermore, we note that since the enrichment functions~\eqref{enrfuns13} satisfy the vanishing conditions at the vertices of $E$, the error bounds established in~\cite{DellAccio:2022:AUE} remain valid.
\end{remark}

\subsection{Convergence analysis}
\noindent
Starting from the enriched finite element $\mathcal{F}=\left(E,V_{h}(E),\Xi_{E}\right)$ with admissible basis functions defined in~\eqref{ggwb1g}, we introduce the enriched projector operator $\pi_{h}^{\mathrm{enr}}:C^{0}(E)\rightarrow V_{h}(E)$ such that
\begin{equation}\label{oper}
\pi_h^{\mathrm{enr}}[g](\mathbf{x}) = \sum_{i=1}^{3} \mathcal{L}_i(g) \varphi_i^{E}(\mathbf{x}) + \sum_{i=1}^{3} \widetilde{\mathcal{L}}_i(g) \varphi^{E}_{i+3}(\mathbf{x}).
\end{equation}
\begin{lemma}\label{lemmaimps}
\noindent
The enriched projector operator reproduces the space of polynomials of degree at most one. 
\end{lemma}
\begin{proof}
\noindent    
We consider a linear polynomial $p\in W_h(E)$ and we write it as a linear combination of the barycentric coordinates $\lambda_{1}^{E},\lambda_{2}^{E}$, $\lambda_{3}^{E}$, i.e.
\begin{equation}\label{linear}
p(\mathbf{x})=\sum_{k=1}^3 p\left(\mathbf{v}_k^{E}\right)\lambda_k^{E}(\mathbf{x}). \end{equation}
\noindent
Then, it results
\begin{equation}\label{lin1}
\pi^{\mathrm{enr}}_h[p](\mathbf{x})=  \sum_{i=1}^{3}\mathcal{L}_i(p)\varphi^{E}_i(\mathbf{x})+ \sum_{i=1}^{3}\widetilde{\mathcal{L}}_i(p)\varphi^{E}_{i+3}(\mathbf{x})=\sum_{i=1}^{3}p\left(\mathbf{v}_{i}^{E}\right)\varphi^{E}_i(\mathbf{x})+ \sum_{i=1}^{3}\widetilde{\mathcal{L}}_i(p)\varphi^{E}_{i+3}(\mathbf{x}).
\end{equation}
\noindent
By the relationship~\eqref{ggwb1g}, it holds
\begin{equation}\label{lin1_bis}
\pi^{\mathrm{enr}}_h[p](\mathbf{x})=\sum_{i=1}^{3}p\left(\mathbf{v}_{i}^{E}\right)\left[\lambda_i^{E}(\mathbf{x})-\frac{1}{2}\sum\limits_{\substack{ %
j=1  \\ j\neq i}}^3\varphi_{j+3}^{E}(\mathbf{x})\right]+ \sum_{i=1}^{3}\widetilde{\mathcal{L}}_i(p)\varphi^{E}_{i+3}(\mathbf{x}).
\end{equation}
\noindent
This expression can be equivalently recast as
\begin{equation}\label{lin1_tris}
\pi^{\mathrm{enr}}_h[p](\mathbf{x})=p(\mathbf{x})-\frac{1}{2}\sum_{i=1}^{3}p\left(\mathbf{v}_{i}^{E}\right)\sum\limits_{\substack{ %
j=1  \\ j\neq i}}^3\varphi_{j+3}^{E}(\mathbf{x})+ \sum_{i=1}^{3}\widetilde{\mathcal{L}}_i(p)\varphi^{E}_{i+3}(\mathbf{x}).
\end{equation}
\noindent
Now, we use the linearity of the functionals $\widetilde{\mathcal{L}}_i$,  $i=1,2,3$, and the Remark~\ref{remimp} to evaluate
\begin{equation}\label{lin2}
\widetilde{\mathcal{L}}_i(p)=\sum_{k=1}^3 p\left(\mathbf{v}_k^{E}\right)\widetilde{\mathcal{L}}_i     \left(\lambda_k^{E}\right)= \frac{1}{2}\sum\limits_{\substack{ %
k=1  \\ k\neq i}}^3 p\left(\mathbf{v}_k^{E}\right).
\end{equation}
\noindent
Combining~\eqref{lin1_tris} and~\eqref{lin2}, we obtain 
\begin{equation*}
    \pi^{\mathrm{enr}}_h[p]=p-\frac{1}{2}\sum_{i=1}^{3}p\left(\mathbf{v}_i^{E}\right)\sum\limits_{\substack{ %
j=1  \\ j\neq i}}^3\varphi^{E}_{j+3}+ \frac{1}{2}\sum_{i=1}^{3}\varphi^{E}_{i+3}\sum\limits_{\substack{ %
k=1  \\ k\neq i}}^3 p\left(\mathbf{v}_k^{E}\right).
\end{equation*}
\noindent
By easy computations, the thesis follows.
\end{proof}\\
\newline
\noindent
In~\cite{DellAccio:2022:AUE}, the authors have shown that, in the case of enrichment functions satisfying the vanishing conditions~\eqref{vancond}, the approximation error associated to $\pi_h^{\mathrm{enr}}$, i.e. 
\begin{equation} \label{eq:error_npe}
 \varepsilon_{h}^{\mathrm{enr}}[g](\mathbf{x})= g(\mathbf{x}) -  \pi_h^{\mathrm{enr}}[g](\mathbf{x}), \quad g\in C^0(E),
\end{equation}
\noindent
can be decomposed as the sum of the error of the linear lagrangian finite element $\varepsilon_{h}^{\mathrm{lin}}$, defined in~\eqref{eq:errorinterp}, and an additional term that depends on both the enrichment functions $\widetilde{\lambda}^{E}_i$ and the enriched linear functionals $\widetilde{\mathcal{L}}_i$, $i=1,2,3$, that is
\begin{equation}\label{erres}
\varepsilon_{h}^{\mathrm{enr}}[g](\mathbf{x})= \varepsilon_{h}^{\mathrm{lin}}[g](\mathbf{x})+ \sum_{\ell=1}^{3} \left\langle {\boldsymbol{\widetilde{\lambda}}}^{E}({\mathbf{x}}),{\mathbf{g}^E_{\ell}}\right\rangle \mathcal{E}_{\ell}(g) \quad \text{with} \quad \mathcal{E}_{\ell}(g)=\frac{1}{2}\sum\limits_{\substack{ %
i=1  \\ i\neq \ell}}^{3}\mathcal{L}_i-\widetilde{\mathcal{L}}_{\ell}.
\end{equation}
\noindent
We introduce the subclass $C^{1,1}(\Omega)$ made by all the functions $g$ that are continuously differentiable with Lipschitz continuous gradient on $\Omega$, i.e. there exists a constant $\rho>0$ such that
\begin{equation}\label{lg}
\left\| \nabla g\left({\mathbf{x}_1}\right)-\nabla g\left({\mathbf{x}_2}\right) \right\|\leq \rho \left\|{\mathbf{x}_1} - {\mathbf{x}_2}\right\|, \quad \forall {\mathbf{x}_1}, {\mathbf{x}_2} \in \Omega.
\end{equation}
\noindent
We call Lipschitz constant for  $\nabla g$ the smallest possible $\rho$ for which the above relationship holds, and we denote it by $L(\nabla g)$. For this kind of functions, the following error estimate has been derived in~\cite{DellAccio:2022:AUE}:
\begin{equation}\label{infplush}
\left\lVert\varepsilon_h^{\mathrm{enr}}[g] \right\rVert_{L^\infty(E)}=\left\lVert g-  \pi_h^{\mathrm{enr}}[g] \right\rVert_{L^\infty(E)}\leq  \frac{L(\nabla g)}{8}\left( 1+ \frac{3}{2}\max_{i=1,2,3} \left\lVert \left\langle {\mathbf{g}}_i^{E}, {\boldsymbol{\widetilde{\lambda}}}^{E}\right\rangle\right\rVert_{L^\infty(E)} \right)h_E^2.
\end{equation}
\noindent
The above error bound serves as the starting point for developing an explicit error bound in $L^{2}$-norm for the enriched finite element $\mathcal{F}$ relative to the sets of enrichment functions that satisfy Theorem~\ref{thmnew}. To this aim, we consider the global enriched projector $\Pi_{h}^{\mathrm{enr}}:C^{0}(\Omega)\rightarrow V_{h}$, such that 
\begin{displaymath}
\Pi_{h}^{\mathrm{enr}}[g]_{|_{E}}=\pi_{h}^{\mathrm{enr}}\left[g_{|_{E}}\right] \quad \forall E\in\mathcal{T}_{h}(\Omega).
\end{displaymath}
\begin{theorem}\label{impthm}
\noindent
Let $\mu,\alpha,\beta>0$ be fixed parameters such that 
\begin{equation*}
\mu+\alpha+\beta=1.    
\end{equation*}
For any $E\in\mathcal{T}_{h}(\Omega)$, let $\mathcal{F}=(E,V_{h}(E),\Xi_{E})$ be the enriched finite element relative to the enrichment functions
\begin{equation}
\widetilde{\lambda}^{E}_k(\mathbf{x})=\omega^{E}_{ \mu ,\alpha,\beta}\left(\mathbf{x}\right)f_{2k-1}\left(\lambda_{k+1}^{E}(\mathbf{x})\right)f_{2k}\left(\lambda_{k+2}^{E}(\mathbf{x})\right), \quad k=1,2,3,
\end{equation}
which satisfy the assumptions of Theorem~\ref{thmnew}. Then, for any $g\in C^{1,1}(\Omega)$, we have
\begin{equation*}
        \left\lVert g-\Pi^{\mathrm{enr}}_h[g]\right\rVert_{L^2(\Omega)}\le \mathcal{C}_{\mu,\alpha,\beta}h^2,
\end{equation*}
where 
\begin{equation}\label{constantCmualphabeta}   \mathcal{C}_{\mu,\alpha,\beta}=\frac{L\left(\nabla g\right)}{8}\left(1+\frac{9(\mu+1)K^2}{2}H_{\alpha,\beta}\right)\sqrt{\left\lvert \Omega\right\rvert},
\end{equation}
is a constant depending only on the parameters $\mu$, $\alpha$, $\beta$, the enrichment functions, the function $g$ and the domain $\Omega$. 
\end{theorem}
\begin{proof}
\noindent
We compute
\begin{equation}\label{conds_1}
    \left\lVert g-\Pi_h^{\mathrm{enr}}[g]\right\rVert_{L^2(\Omega)}^2=\int_{\Omega}\left\lvert g(\mathbf{x})-\Pi_h^{\mathrm{enr}}[g](\mathbf{x})\right\rvert^2 d \mathbf{x}=\sum_{E\in\mathcal{T}_h(\Omega)}\int_{E}\left\lvert g(\mathbf{x})-\pi_h^{\mathrm{enr}}[g](\mathbf{x})\right\rvert^2 d \mathbf{x}.
\end{equation}
\noindent
By virtue of the local error bound \eqref{infplush}, we obtain
\begin{eqnarray}\label{conds}
\left\lVert g-\pi_h^{\mathrm{enr}}[g]\right\rVert_{L^2(\Omega)}^2&\le& \sum_{E\in\mathcal{T}_h(\Omega)}\left(\frac{L(\nabla g)}{8}\left(1+\frac{3}{2}\max_{i=1,2,3}\left\lVert \left\langle \mathbf{g}^{E}_i,\boldsymbol{\widetilde{\lambda}}^{E} \right\rangle\right\rVert_{L^\infty(E)}\right)\right)^2 h_E^4 \left\lvert E \right\rvert\notag\leq\\
&\le& h^4 \left(\frac{L(\nabla g)}{8}\right)^2 \sum_{E\in\mathcal{T}_h(\Omega)}\left(1+\frac{3}{2}\max_{i=1,2,3}\left\lVert \left\langle \mathbf{g}_i^{E},\boldsymbol{\widetilde{\lambda}}^{E} \right\rangle\right\rVert_{L^\infty(E)}\right)^2 \left\lvert E \right\rvert.
\end{eqnarray}
\noindent
Now, for any $E\in \mathcal{T}_h(\Omega)$, we have 
\begin{equation}\label{provssa}
\left\lVert \left\langle \mathbf{g}_i^{E},\boldsymbol{\widetilde{\lambda}}^{E} \right\rangle\right\rVert_{L^\infty(E)}=\left\lVert g_{i,1}^{E} \widetilde{\lambda}_1^{E}+g_{i,2}^{E} \widetilde{\lambda}_2^{E}+g_{i,3}^{E}\widetilde{\lambda}_3^{E} \right\rVert_{L^\infty(E)}\leq\max_{i,j=1,2,3}\left\lvert g_{i,j}^{E}\right\rvert \left(\sum_{k=1}^3\left\lVert \widetilde{\lambda}_k^{E}\right\rVert_{L^\infty(E)}\right).
\end{equation}
\noindent
Since the weight function~\eqref{defofh} is positive, by~\eqref{partition_of_unity} and the application of the generalized arithmetic-geometric mean inequality (see~\cite{chuan1990note}), for any $\mathbf{x}\in E$, it holds
\begin{eqnarray*}
\omega_{\mu,\alpha,\beta}^{E}(\mathbf{x})&=&   \sum_{j=1}^{3}\left(1-\lambda_{j}^{E}\left(\mathbf{x}\right)\right)^\mu\left(\lambda^{E}_{j+1}\left(\mathbf{x}\right)\right)^{\alpha}\left(\lambda^{E}_{j+2}\left(\mathbf{x}\right)\right)^{\beta}\le\\ &\le&  \sum_{j=1}^{3}\left(\mu\left(1-\lambda_{j}^{E}\left(\mathbf{x}\right)\right)+\alpha\lambda^{E}_{j+1}\left(\mathbf{x}\right)+\beta\lambda^{E}_{j+2}\left(\mathbf{x}\right)\right)=\\ 
   &=& \left(\mu\sum_{j=1}^{3}\left(1-\lambda_{j}^{E}\left(\mathbf{x}\right)\right)+\alpha \sum_{j=1}^{3}\lambda^{E}_{j+1}\left(\mathbf{x}\right)+\beta\sum_{j=1}^{3}\lambda^{E}_{j+2}\left(\mathbf{x}\right)\right)=\\
   &=& 2\mu+\alpha+\beta=\mu+1. 
\end{eqnarray*}
\noindent
As a consequence, we have
\begin{equation*}
    \max_{\mathbf{x}\in E} \ \omega_{\mu,\alpha,\beta}^{E}(\mathbf{x})\le \mu+1.
\end{equation*}
\noindent
From the above relationship, it follows that
\begin{eqnarray}
    \sum_{k=1}^3 \left\lVert \widetilde{\lambda}_k^{E}\right\rVert_{L^\infty(E)}&=&\max_{\mathbf{x}\in E}\sum_{k=1}^3\left\lvert \omega^{E}_{ \mu ,\alpha,\beta}(\mathbf{x})f_{2i-1}\left(\lambda_{i+1}^{E}(\mathbf{x})\right)f_{2i}\left(\lambda_{i+2}^{E}(\mathbf{x})\right) \right\rvert \le \notag\\
    &\le& (\mu+1)\max_{\mathbf{x}\in E}\sum_{k=1}^3 \left\lvert f_{2i-1}\left(\lambda_{i+1}^{E}(\mathbf{x})\right)f_{2i}\left(\lambda_{i+2}^{E}(\mathbf{x})\right) \right\rvert\le \notag\\ \label{provss}
    &\le& 3(\mu+1) K^2,
\end{eqnarray}
\noindent
where the constant $K$ is given by
\begin{equation}\label{boundvarphi}
   K=\max_{\ell=1,\dots,6}\left\lVert f_{\ell} \right\rVert_{L^\infty([0,1])}.  
\end{equation}
\noindent
As the enrichment functions~\eqref{enrfuns13} satisfy the vanishing conditions at the vertices of $E$, the matrix $\mathbb{G}$ takes the form given in~\eqref{matN}. Therefore, its inverse is the diagonal matrix
\begin{equation}\label{invmatN}
   \mathbb{G}^{-1}= \left[\mathbf{g}^{E}_1 \mathbf{g}_2^{E} \mathbf{g}_3^{E}\right]=\begin{pmatrix}
\left(\widetilde{\mathcal{L}}_1\left(\widetilde{\lambda}_1^{E}\right)\right)^{-1} & 0&0 \\
0 & \left(\widetilde{\mathcal{L}}_2\left(\widetilde{\lambda}_2^{E}\right)\right)^{-1}&0 \\
0 & 0&\left(\widetilde{\mathcal{L}}_3\left(\widetilde{\lambda}_3^{E}\right)\right)^{-1}
\end{pmatrix}. 
\end{equation}
\noindent
Since $\left\lvert g_{i,j}^{E}\right\rvert=0$ for $i\neq j$, we consider
\begin{equation}\label{Hmu_1}   
\left\lvert g_{i,i}^{E}\right\rvert= \frac{1}{\left\lvert\widetilde{\mathcal{L}}_i\left(\widetilde{\lambda}_i^{E}\right)\right\rvert} = \frac{1}{\left\lvert\displaystyle\frac{1}{ \left\lvert \mathbf{e}_i^{E}\right\rvert}\int_{ \mathbf{e}_i^{E}} \widetilde{\lambda}^{E}_i(\mathbf{x})d \mathbf{x}\right\rvert}.
\end{equation}
\noindent
Using the following parametrization of the edge $\mathbf{e}^{E}_{i}$
\begin{eqnarray*}
t&\rightarrow& t {\mathbf{v}}^{E}_{i+1}+(1-t){\mathbf{v}}^{E}_{i+2}, \quad t\in[0,1]
\end{eqnarray*}
\noindent
the relationship~\eqref{Hmu_1} can be recast in the following form
\begin{equation*}   
\left\lvert g_{i,i}^{E}\right\rvert=  \frac{1}{\left\lvert\displaystyle\int_{0}^1 \omega_{\alpha,\beta}(t)f_{2i-1}\left(\lambda_{i+1}^{E}\left(t\mathbf{v}_{i+1}^{E}+(1-t)\mathbf{v}_{i+2}^{E}\right)\right)f_{2i}\left(\lambda_{i+2}^E\left(t\mathbf{v}_{i+1}^{E}+(1-t)\mathbf{v}_{i+2}^{E}\right)\right) d t\right\rvert}.
\end{equation*}
\noindent
Taking into account the relationships~\eqref{deltaprop} and~\eqref{prop21}, we get
\begin{equation}\label{Hmu}   
\left\lvert g_{i,i}^{E}\right\rvert=  \frac{1}{\left\lvert\displaystyle\int_{0}^1 \omega_{\alpha,\beta}(t)f_{2i-1}\left(t\right)f_{2i}\left(1-t\right) d t\right\rvert} :=H_{{i,\alpha,\beta}}
\end{equation}
\noindent
and consequently we have
\begin{equation}\label{altra}
    \max_{i=1,2,3}\left\lvert g_{i,i}^E \right\rvert=\max_{i=1,2,3} H_{i,\alpha,\beta}=H_{\alpha,\beta},
\end{equation}
\noindent
where the constant $H_{\alpha,\beta}$ depends only on the enrichment functions and not on the triangulation.
Combining~\eqref{provssa},~\eqref{provss}, and~\eqref{altra}, we can conclude that
\begin{equation*}
    \max_{i=1,2,3}\left\lVert \left\langle \mathbf{g}_i^{E},\boldsymbol{\widetilde{\lambda}}^{E} \right\rangle\right\rVert_{L^\infty(E)}\le 3H_{\alpha,\beta}(\mu+1) K^2.
\end{equation*}
\noindent
Finally, it results
\begin{eqnarray}\label{ssssa}
    \left\lVert g-\Pi^{\mathrm{enr}}_h[g]\right\rVert_{L^2(\Omega)}^2&\le& h^4 \left(\frac{L(\nabla g)}{8}\right)^2 \left(1+\frac{9(\mu+1) K^2}{2}H_{\alpha,\beta}\right)^2\sum_{E\in\mathcal{T}_h(\Omega)} \left\lvert E \right\rvert= \notag\\ &=&h^4 \left(\frac{L(\nabla g)}{8}\right)^2 \left(1+\frac{9(\mu+1) K^2}{2}H_{\alpha,\beta}\right)^2 \left\lvert \Omega \right\rvert.
\end{eqnarray}
Taking the square root of both sides of~\eqref{ssssa}, we obtain the thesis. 
\end{proof}\\
\\
\noindent
In the following theorem, we determine the \textit{optimal} values of $\mu$, $\alpha$, $\beta$ for the set of admissible weighted enrichment functions $\mathcal{E}_{15}$. These values minimize the constant $\mathcal{C}_{\mu,\alpha,\beta}$ given in the previous theorem.

\begin{theorem}\label{thmimss}
Let $\alpha_1,\beta_1\ge 0$ be fixed parameters. For the set of admissible weighted enrichment functions $\mathcal{E}_{15}$, the values that minimize the constant $\mathcal{C}_{\mu,\alpha,\beta}$ are 
   \begin{equation*}
\begin{cases}
(\mu,\alpha,\beta)=(0,1,0) \quad \text{ if } \quad \alpha_1\ge\beta_1 \\
(\mu,\alpha,\beta)=(0,0,1) \quad \text{ if } \quad \beta_1\ge\alpha_1. \\
\end{cases}
 \end{equation*}
\end{theorem}
\begin{proof}
\noindent
From the bound~\eqref{constantCmualphabeta}, it is clear that we have to put $\mu=0$, in order to minimize the constant $\mathcal{C}_{\mu,\alpha,\beta}$. To obtain the thesis, we aim at determining the value of the parameters $\alpha$ and $\beta$, with $\alpha+\beta=1$, that minimizes the constant, i.e. that maximizes the integral
\begin{eqnarray*}
I(\alpha,\beta)&=&\int_0^1 t^{\alpha+\alpha_1}(1-t)^{\beta+\beta_1} d t=B\left(\alpha+\alpha_1+1,\beta+\beta_1+1\right)\\ &=&\frac{\Gamma\left(\alpha+\alpha_1+1\right)\Gamma\left(\beta+\beta_1+1\right)}{\Gamma\left(\alpha+\beta+\alpha_1+\beta_1+2\right)},
\end{eqnarray*}
\noindent
where $B(\cdot,\cdot)$ and $\Gamma(\cdot)$ are the beta and gamma functions, respectively (see~\cite{Abramowitz:1948:HOM}).
Since $\alpha+\beta=1$, we can write
\begin{equation*}
I(\alpha,1-\alpha)=\frac{\Gamma\left(\alpha+\alpha_1+1\right)\Gamma\left(\beta_1-\alpha+2\right)}{\Gamma\left(\alpha_1+\beta_1+3\right)}. 
\end{equation*}
\noindent
In the above relationship, since the denominator is constant, we maximize the univariate function $F:[0,1]\rightarrow\mathbf{R}$ such that
\begin{equation*}
F(\alpha)=\Gamma\left(\alpha+\alpha_1+1\right)\Gamma\left(\beta_1-\alpha+2\right).
\end{equation*}
\noindent
At this stage, we introduce the function
\begin{equation*}
   Z(\alpha):=\ln \left(F(\alpha)\right)=\ln\left(\Gamma\left(\alpha+\alpha_1+1\right)\right)+\ln\left(\Gamma\left(\beta_1-\alpha+2\right)\right)
\end{equation*}
\noindent
and we compute
\begin{equation*}
    Z^{\prime}(\alpha)=\frac{d}{d\alpha}\ln\left(\Gamma(\alpha)\right)=\Psi\left(\alpha+\alpha_1+1\right)-\Psi\left(\beta_1-\alpha+2\right),
\end{equation*}
\noindent
where $\Psi(\cdot)$ denotes the digamma function (see~\cite{Abramowitz:1948:HOM}). Setting $Z^{\prime}(\alpha)=0$ yields
\begin{equation*}
\Psi\left(\alpha+\alpha_1+1\right)=\Psi\left(\beta_1-\alpha+2\right).    
\end{equation*}
\noindent
Since the digamma function is strictly increasing, it follows that
\begin{equation*}
\alpha+\alpha_1+1=\beta_1-\alpha+2,    
\end{equation*}
\noindent
which leads to
\begin{equation*}
\alpha=\frac{\beta_1-\alpha_1+1}{2}.  
\end{equation*}
\noindent
Calculating the second derivative $Z^{\prime\prime}(\alpha)$ and using the polygamma function (see~\cite{Abramowitz:1948:HOM}), we find that the above value of $\alpha$ is a minimum for $Z(\alpha)$. 
Therefore, the maximum must be assumed at one of the endpoints. In particular, if $\alpha=0$ (and $\beta=1$), then
\begin{equation*}
F(0)=\Gamma\left(\alpha_1+1\right)\Gamma\left(\beta_1+2\right)=\left(\beta_1+1\right)\Gamma\left(\alpha_1+1\right)\Gamma\left(\beta_1+1\right). 
\end{equation*}
\noindent
On the other hand, if $\alpha=1$ (and $\beta=0$), we obtain
\begin{equation*}
F(1)={\Gamma\left(\alpha_1+2\right)\Gamma\left(\beta_1+1\right)}={\left(\alpha_1+1\right)\Gamma\left(\alpha_1+1\right)\Gamma\left(\beta_1+1\right)}. 
\end{equation*}
\noindent
Comparing these two values and simplify, the result follows.
\end{proof}\\

\noindent
Unfortunately, for certain sets of weighted enrichment functions, the explicit computation of the value of $H_{\alpha,\beta}$ is not straightforward. Therefore, in the next result, we provide an estimate of this constant for the family of weighted enrichment functions introduced in Theorem~\ref{thmnew} and determine an appropriate choice for the parameters $\alpha$ and $\beta$. 
\begin{corollary}
\noindent
For the family of weighted enrichment functions introduced in Theorem~\ref{thmnew}, the values
\begin{equation*}
(\mu,\alpha,\beta)=(0,1,0) 
\end{equation*}
or 
\begin{equation*}
(\mu,\alpha,\beta)=(0,0,1), 
\end{equation*}
provide a \textit{quasi-optimal} constant $\mathcal{C}_{\mu,\alpha,\beta}$.
\end{corollary}
\begin{proof}
From the bound~\eqref{constantCmualphabeta}, it is clear that we have to put $\mu=0$, in order to minimize the constant $\mathcal{C}_{\mu,\alpha,\beta}$. Now, we aim to determinig the coefficients $\alpha$ and $\beta$ satisfying
\begin{equation}\label{cond111}
     \alpha+\beta=1
 \end{equation}
 which minimize the value $H_{\alpha,\beta}$, or equivalently which maximize the integral
 \begin{eqnarray*}
     I(\alpha,1-\alpha)&=&\left\lvert\int_0^1 t^{\alpha}(1-t)^{1-\alpha} f_{2i-1}\left(t\right)f_{2i}\left(1-t\right) d t\right\rvert \le\notag\\
     &\le& \int_0^1 \left\lvert t^{\alpha}(1-t)^{1-\alpha} f_{2i-1}\left(t\right)f_{2i}\left(1-t\right) \right\rvert d t \notag\le\\
     &\le& K^2 \int_0^1  t^{\alpha}(1-t)^{1-\alpha} d t=K^2 B(\alpha+1,2-\alpha),
 \end{eqnarray*}
where $K$ is defined in~\eqref{boundvarphi}. The result follows by Theorem~\ref{thmimss} applied to the case $\alpha_1=\beta_1=0$.
\end{proof}\\

\noindent
Finally, we point out that Lemma~\ref{lemmaimps} allows us to derive the following bound for the approximation error.
\begin{theorem}
\noindent
For any $E\in\mathcal{T}_{h}(\Omega)$, let $\mathcal{F}=(E,V_{h}(E),\Xi_{E})$ be the enriched finite element relative to a set of admissible enrichment functions. Then, for any $g\in H^{2}(\Omega)$, we have
\begin{equation*}
        \left\| g-\Pi^{\mathrm{enr}}_h[g]\right\|_{H^{\ell}(\Omega)}\le \mathcal{C}h^{2-\ell}\|g\|_{H^{2}(\Omega)}, \qquad \ell=0,1
\end{equation*}
where $\mathcal{C}>0$ is a constant independent of $h$.
\end{theorem}
\noindent
Using the above theorem and standard arguments of elliptic operators, it is possible to obtain the following result.
\begin{theorem}
\noindent
Let $u\in V$be the exact solution of the variational problem \eqref{weak_problem} and $u_{h}\in V_{h}$ its approximate solution using the enriched linear lagrangian finite element method. If $u\in H^{2}(\Omega)$, then the following a priori error estimate holds
\begin{equation*}
        \left\| u-u_h\right\|_{H^{1}(\Omega)}\le \frac{\mathcal{C}}{1+C_{\Omega}^{2}}h\|u\|_{H^{2}(\Omega)}, 
\end{equation*}
where $\mathcal{C}>0$ is a constant independent of $h$ and $u$.
\end{theorem}
\section{Numerical results}\label{sec4}
\noindent
In this section, we perform some numerical experiments to demonstrate the effectiveness of our enriched finite element method, in both the cases of weighted and unweighted enrichment strategies.\\
In particular, we consider four Poisson problems and, for each of them, we compare the trends of the error in the energy semi-norm for the standard linear lagrangian finite element, denoted by
\begin{equation*}
\epsilon^{\mathrm{lin}}(u)=\sqrt{\int\limits_{\Omega} \left[\nabla \left(u-u_h^{\mathrm{lin}}\right)(\mathbf{x}) \right]^{2} d \mathbf{x}}
\end{equation*}
\noindent
with the error in the energy semi-norm for the enriched linear lagrangian finite element, denoted by
\begin{equation*}
\epsilon_{\mathcal{E}}^{\mathrm{enr}}(u)=\sqrt{\int\limits_{\Omega} \left[\nabla \left(u-u_h^{\mathrm{enr}}\right)(\mathbf{x}) \right]^{2} d \mathbf{x}}
\end{equation*}
\noindent
where $\mathcal{E}=\mathcal{E}_{10},\mathcal{E}_{11},\mathcal{E}_{12}$. To develop a convergence analysis, we start by choosing a coarse Friedrich-Keller triangulation (see~\cite{Knabner}) of the unit square $\Omega:=[0,1]^{2}$, made by 32 triangles as shown in the left plot of Fig.~\ref{triangles}. All successive refinements are obtained by halving each side of its elements.

\begin{figure}[ht!]
  \centering
\includegraphics[width=0.19\textwidth]{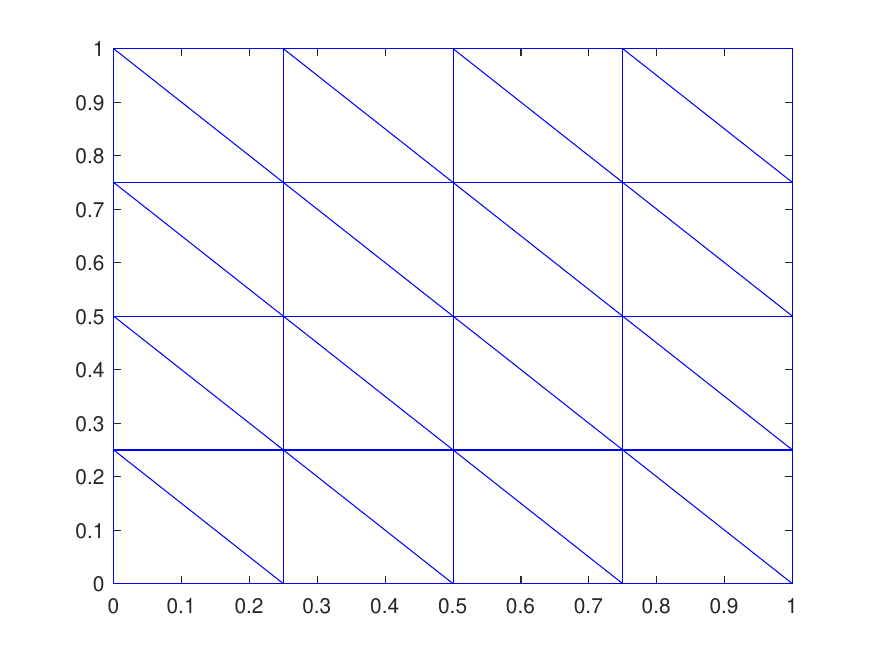} 
\includegraphics[width=0.19\textwidth]{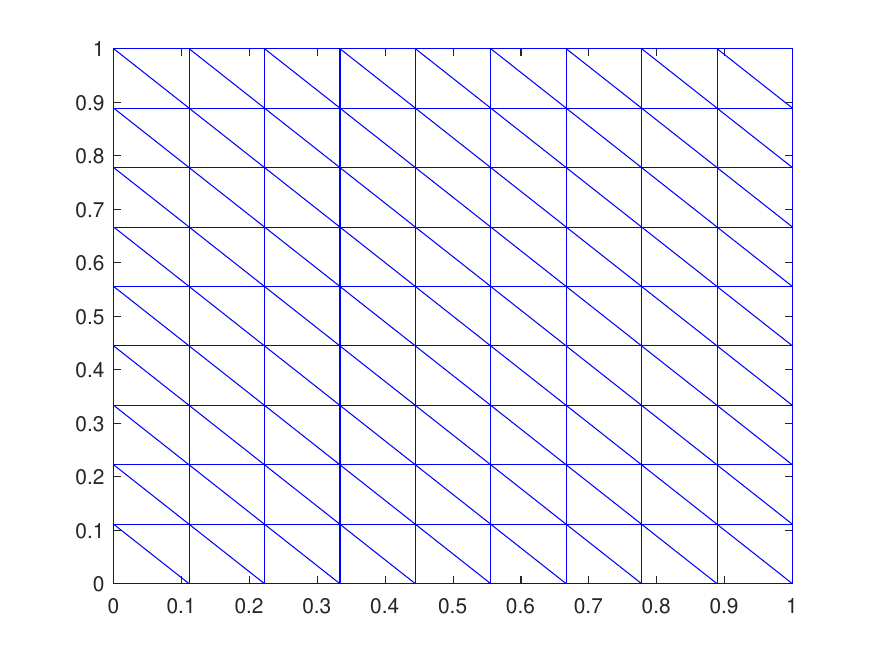} 
\includegraphics[width=0.19\textwidth]{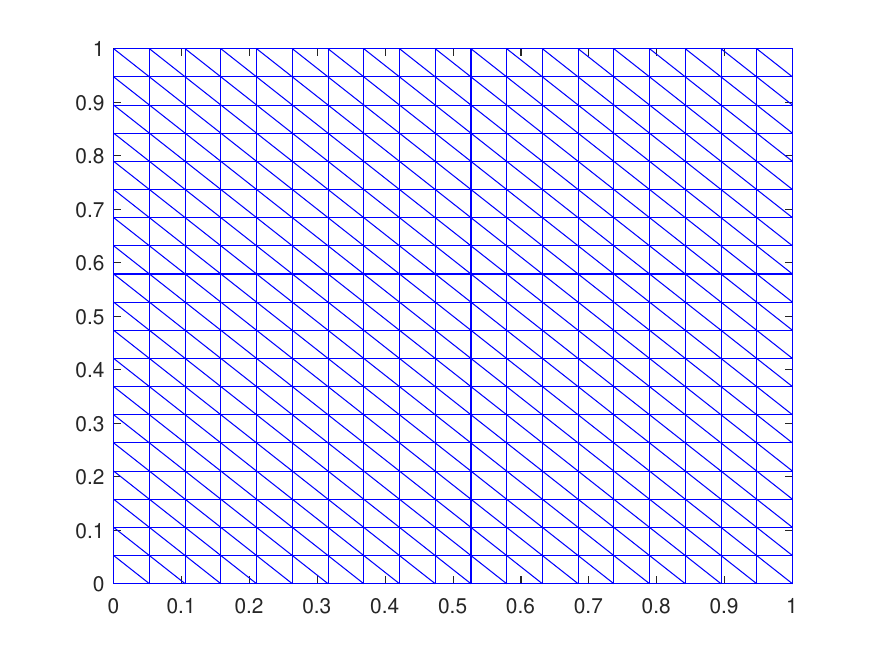} 
\includegraphics[width=0.19\textwidth]{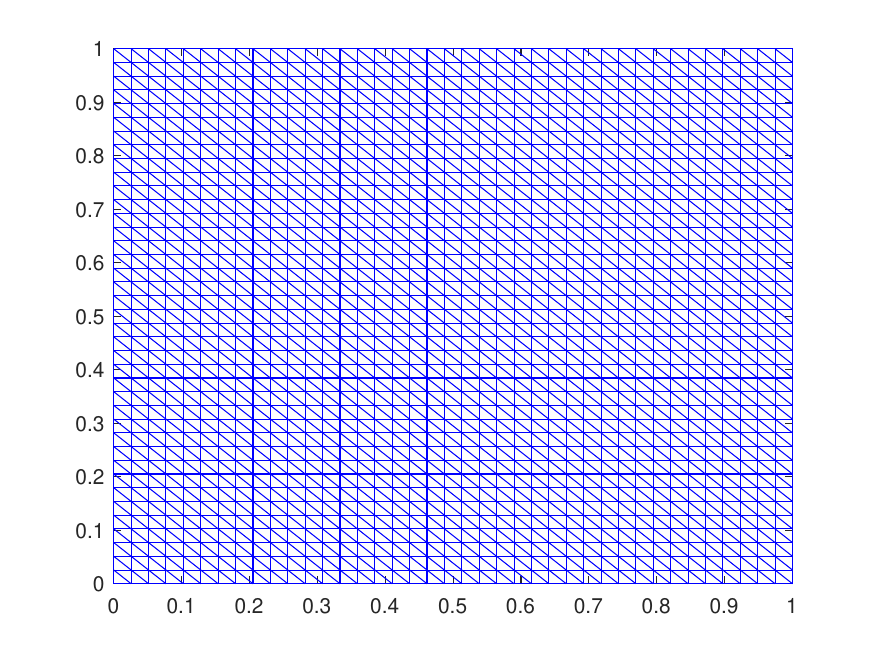} 
\includegraphics[width=0.19\textwidth]{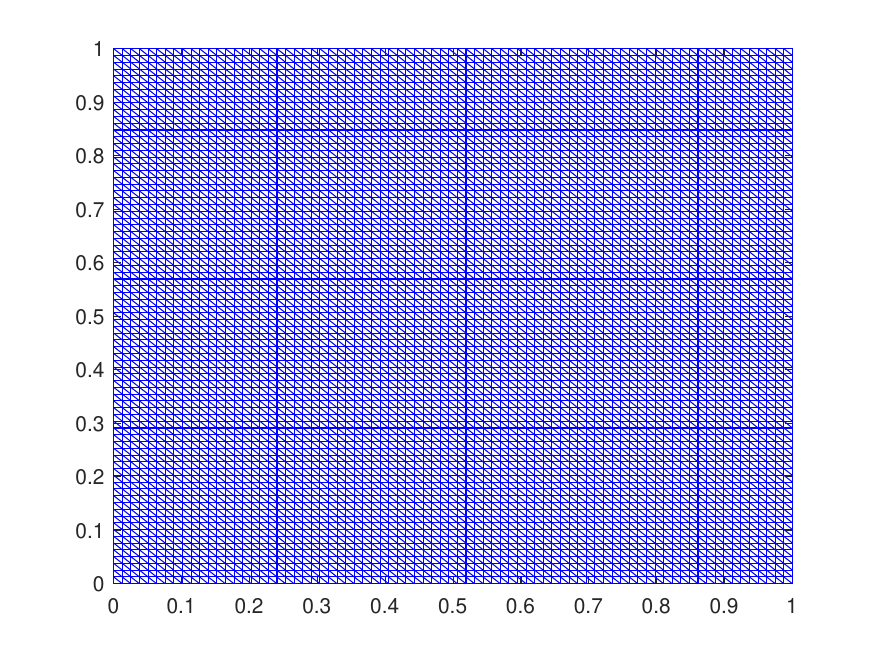} 
 \caption{Successive regular Friedrichs-Keller triangulations of the unit square $\Omega=[0,1]^{2}$.}
 \label{triangles}
\end{figure}
\noindent
First, we focus on the unweighted enrichment strategy, i.e. $\mu=\alpha=\beta=0$. Although the solutions to the considered problems exhibit different characteristics, the enrichment strategy that produces the most significant improvement is $\mathcal{E}_{10}$, while those provided by $\mathcal{E}_{11}$ and $\mathcal{E}_{12}$ perform comparably. We remark that all these enrichment strategies outperform the linear lagrangian finite element. In Fig.~\ref{condplot}, we display the trends of the condition numbers of the corresponding system matrices, showing that the condition numbers for the enriched finite elements based on $\mathcal{E}_{10}$ (denoted by $\kappa^{\mathrm{enr}}_{\mathcal{E}_{10}}$), $\mathcal{E}_{11}$ (denoted by $\kappa^{\mathrm{enr}}_{\mathcal{E}_{11}}$) and $\mathcal{E}_{12}$ (denoted by $\kappa^{\mathrm{enr}}_{\mathcal{E}_{12}}$) are comparable to those produced by the linear lagrangian finite element (denoted by $\kappa^{\mathrm{lin}}$).
To further illustrate the effectiveness of the proposed approach, in Fig.~\ref{plotrec} we reconstruct the solutions of the four Poisson problems using the enriched finite element based on $\mathcal{E}_{10}$ on the last three Friedrichs-Keller meshes of Fig.~\ref{triangles}.\\
Then, we set $\mu=\beta=0$, $\alpha=1$ 
and we consider the weight function $\omega_{\mu,\alpha,\beta}^{E}$ defined in~\eqref{defofh}.
In analogy to the unweighted approach, we consider the same four Poisson problems and compare the trends of the errors in the energy norm for the enriched finite elements based on the sets of admissible weighted enrichment functions $\mathcal{E}_{10}$, $\mathcal{E}_{11}$ and $\mathcal{E}_{12}$. The results of these experiments are shown in Fig.~\ref{Fig1w} and Fig.~\ref{Fig2w}.  From these plots, it is evident that the weighted strategy produces similar results compared to the unweighted case. As observed previously, the weighted enrichment strategy that yields the most significant improvement is $\mathcal{E}_{10}$, while those provided by $\mathcal{E}_{11}$ and $\mathcal{E}_{12}$ perform comparably.\\
All numerical computations have been performed on a PC with Intel\textsuperscript{\textregistered} Core\textsuperscript{TM} i7-7700 (3.60 GHz) by means of Matlab\textsuperscript{\textregistered} codes.\
All experiments were carried out using \texttt{MATLAB}. To achieve high-precision integral evaluations, we employed Gaussian quadrature with a suitably chosen number of quadrature points.
\subsubsection{Problem 1}
\noindent
We consider the following Laplace boundary-value problem
\begin{subequations}
  \begin{empheq}[left=\empheqlbrace]{align} &-\Delta u_1(\textbf{x})=f_1(\textbf{x})& &\textbf{x}\in \Omega\\
 &u_1(\textbf{x})=0& 	 					 &\textbf{x}\in\Gamma, 
  \end{empheq}
\end{subequations}
where the exact solution is given by 
\begin{equation*}
   u_1(\mathbf{x})=\sin(2\pi x_{1})\sin(2\pi x_{2}), 
   \end{equation*}
 and the source term is defined as
 \begin{equation*}
f_1(\mathbf{x})=8\pi^2\sin\left(2\pi x_{1}\right)\sin(2\pi x_{2}).   
\end{equation*}
The numerical results for this problem are presented in Fig.~\ref{Fig1} (left).

\subsubsection{Problem 2}
\noindent
We consider the following Laplace boundary-value problem
\begin{subequations}
  \begin{empheq}[left=\empheqlbrace]{align} &-\Delta u_2(\textbf{x})=f_2(\textbf{x})& &\textbf{x}\in\Omega\\
 &u_2(\textbf{x})=0& 	 					 &\textbf{x}\in\Gamma, 
  \end{empheq}
\end{subequations}
where the exact solution is given by 
\begin{equation*}
     u_2(\mathbf{x})=(e^{x_{1}(1-x_{1})}-1)\sin(2 \pi x_{2}),
   \end{equation*}
 and the source term is defined as
 \begin{equation*}
f_2(\mathbf{x})= 2 e^{-x_{1}(x_{1}-1)} \sin(2\pi x_{2}) + 4\pi^2 \sin(2\pi x_{2}) \left( e^{-x_{1}(x_{1}-1)} - 1 \right) - e^{-x_{1}(x_{1}-1)} \sin(2\pi x_{2})(2x_{1}-1)^2.  
\end{equation*}
The numerical results for this problem are presented in Fig.~\ref{Fig1} (right).

\subsubsection{Problem 3}
\noindent
We consider the following Laplace boundary-value problem
\begin{subequations}
  \begin{empheq}[left=\empheqlbrace]{align} &-\Delta u_3(\textbf{x})=f_3(\textbf{x})& &\textbf{x}\in\Omega\\
 &u_3(\textbf{x})=0& 	 					 &\textbf{x}\in\Gamma, 
  \end{empheq}
\end{subequations}
where the exact solution is given by 
\begin{equation*}
   u_3(\mathbf{x})=(e^{x_{1}(1-x_{1})}-1)(e^{x_{2}(1-x_{2})}-1),
\end{equation*}
and the source term is defined as
\begin{equation*}
\begin{split}
f_3(\mathbf{x}) =\;& 2e^{-x_{1}(x_{1}-1)}\left(e^{-x_{2}(x_{2}-1)}-1\right) 
+ 2e^{-x_{2}(x_{2}-1)}\left(e^{-x_{1}(x_{1}-1)}-1\right) \\
& - e^{-x_{1}(x_{1}-1)}\,(2x_{1}-1)^2\left(e^{-x_{2}(x_{2}-1)}-1\right)
- e^{-x_{2}(x_{2}-1)}\,(2x_{2}-1)^2\left(e^{-x_{1}(x_{1}-1)}-1\right).
\end{split}
\end{equation*}
The numerical results for this problem are presented in Fig.~\ref{Fig2} (left).

\subsubsection{Problem 4}
\noindent
We consider the following Laplace boundary-value problem
\begin{subequations}
  \begin{empheq}[left=\empheqlbrace]{align} &-\Delta u_4(\textbf{x})=f_4(\textbf{x})& &\textbf{x}\in\Omega\\
 &u_4(\textbf{x})=0& 	 					 &\textbf{x}\in\Gamma, 
  \end{empheq}
\end{subequations}
where the exact solution is given by 
\begin{equation*}
    u_4(\mathbf{x})=x_{1}x_{2}(1-x_{1})(1-x_{2}),
\end{equation*}
and the source term is defined as
\begin{equation*}
f_4(\mathbf{x}) =-2x_{1}(x_{1} - 1)-2x_{2}(x_{2} - 1).
\end{equation*}
The numerical results for this problem are presented in Fig.~\ref{Fig2} (right).

\begin{figure}
  \centering
\includegraphics[width=0.32\textwidth]{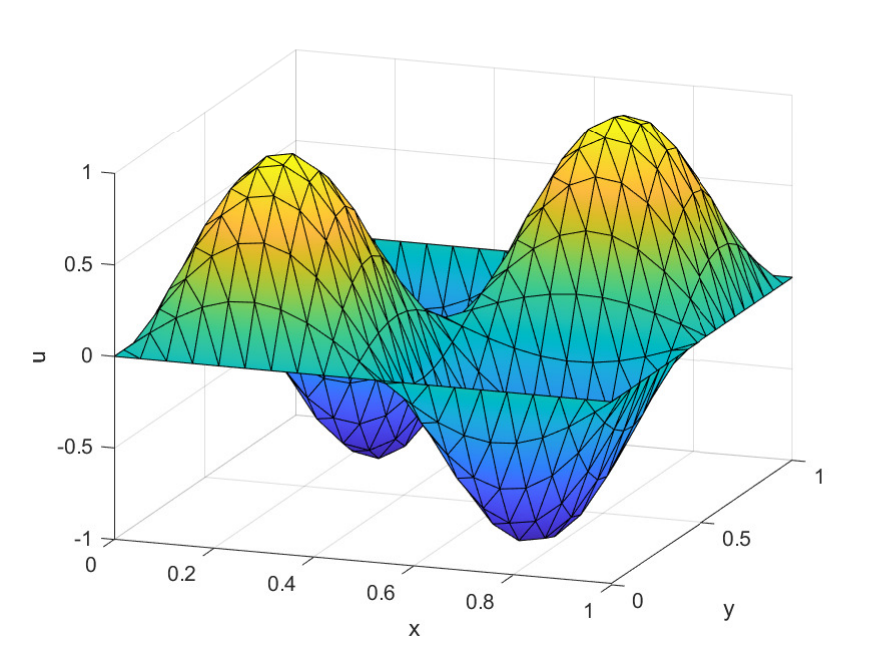} 
\includegraphics[width=0.32\textwidth]{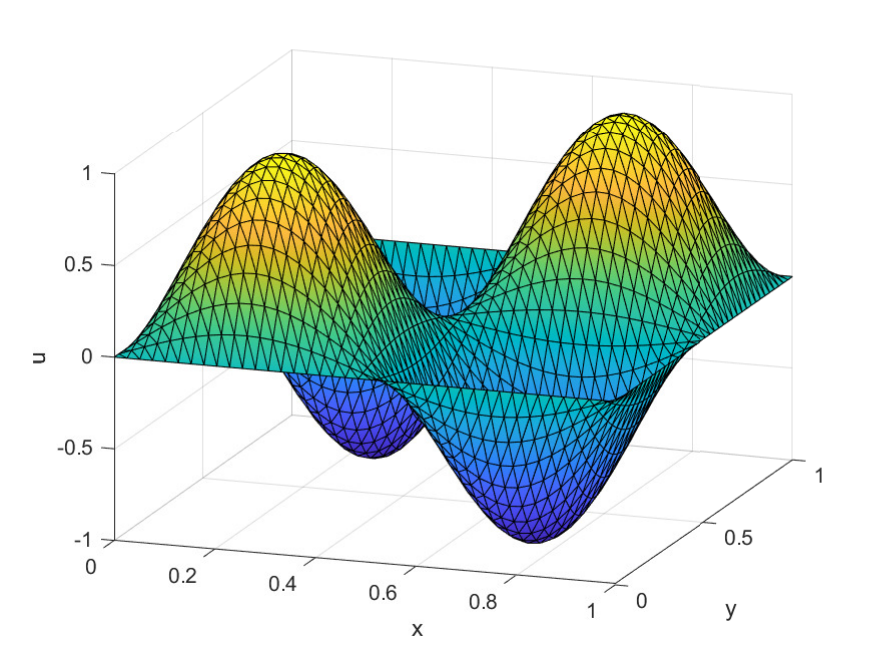} 
\includegraphics[width=0.32\textwidth]{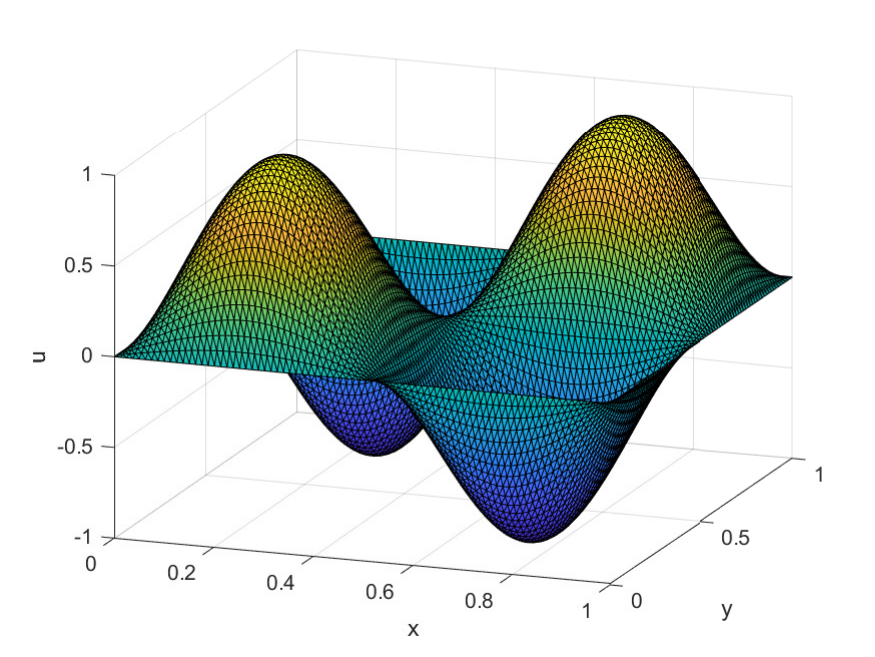}
\includegraphics[width=0.32\textwidth]{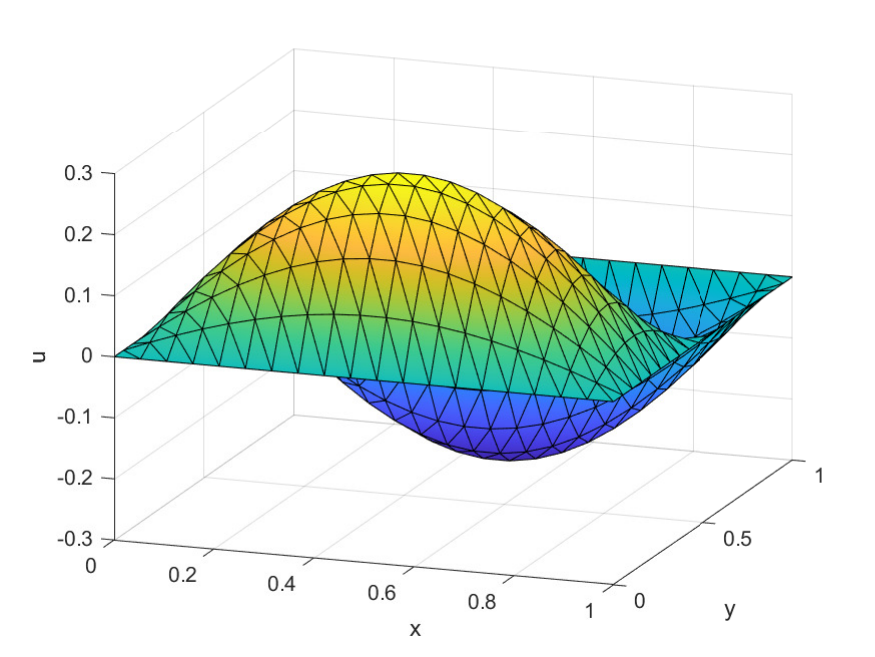} 
\includegraphics[width=0.32\textwidth]{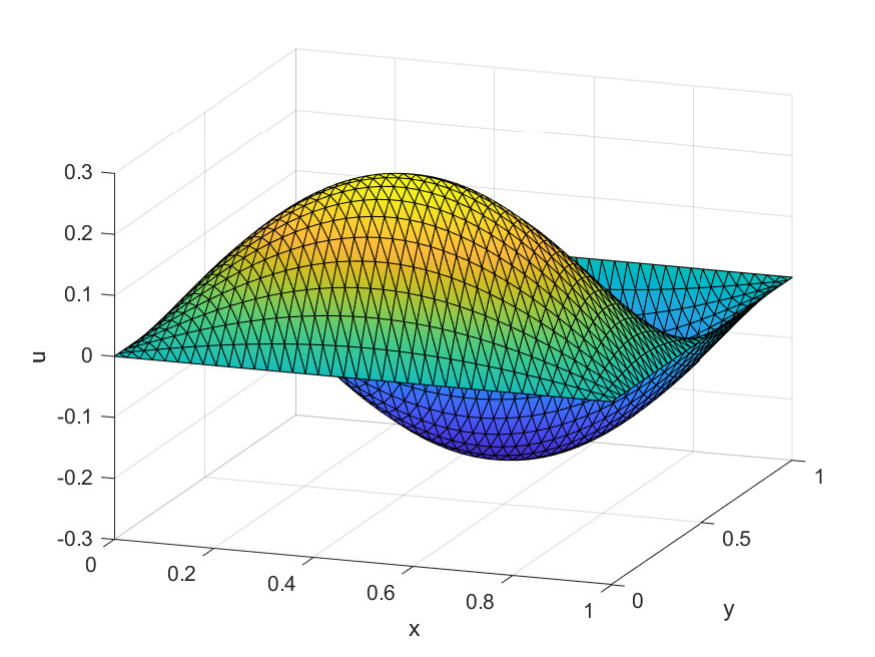} 
\includegraphics[width=0.32\textwidth]{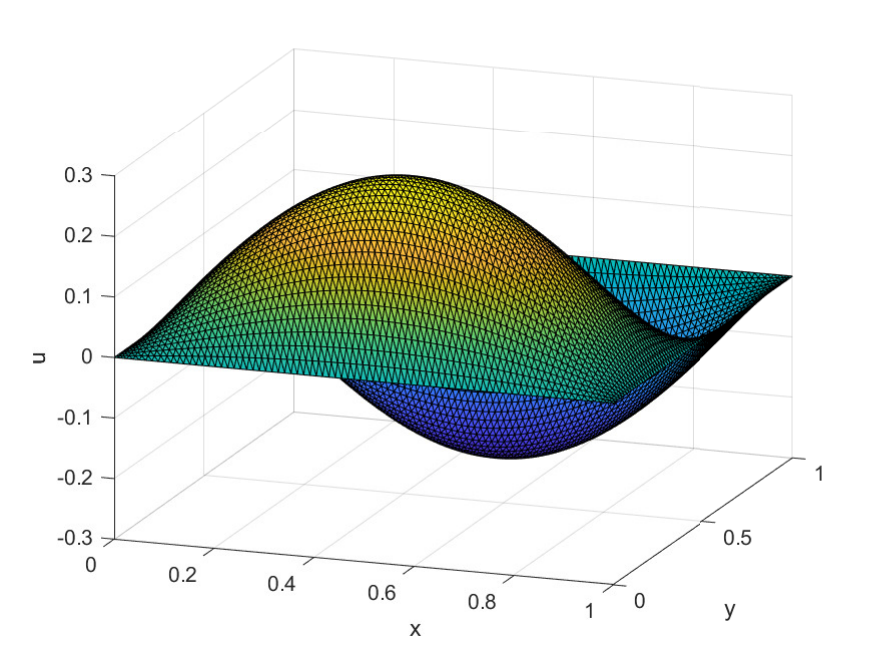}
\includegraphics[width=0.32\textwidth]{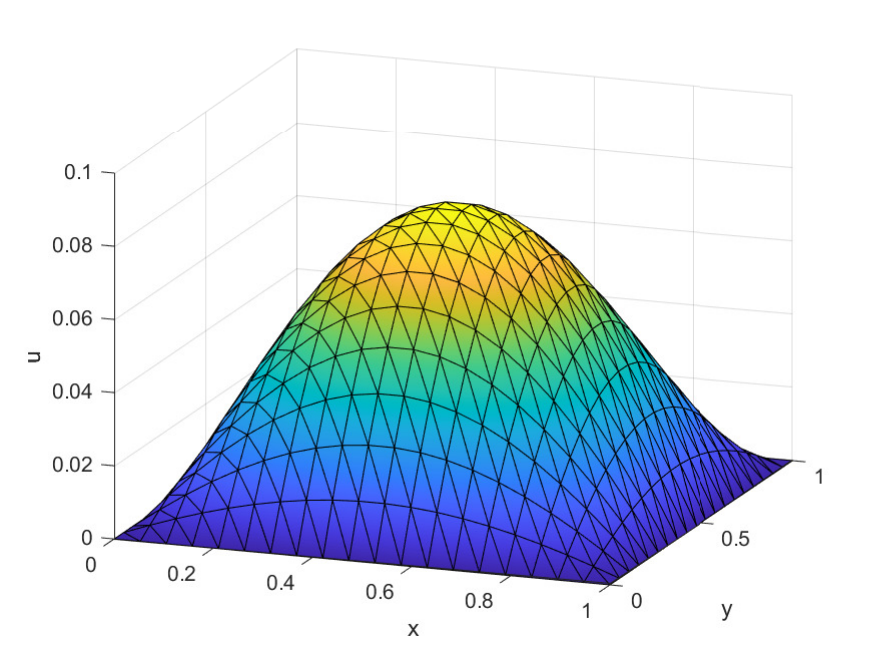} 
\includegraphics[width=0.32\textwidth]{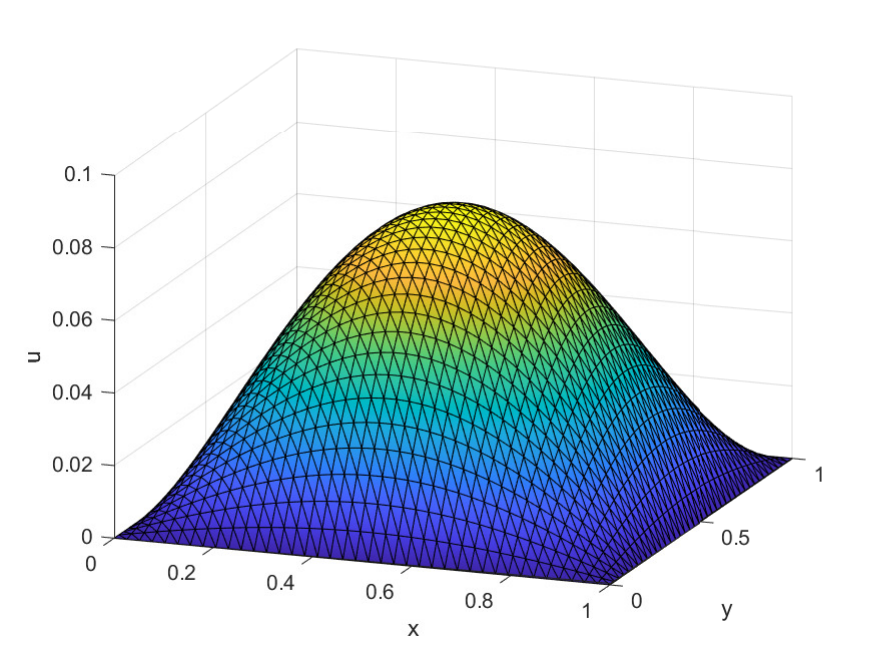} 
\includegraphics[width=0.32\textwidth]{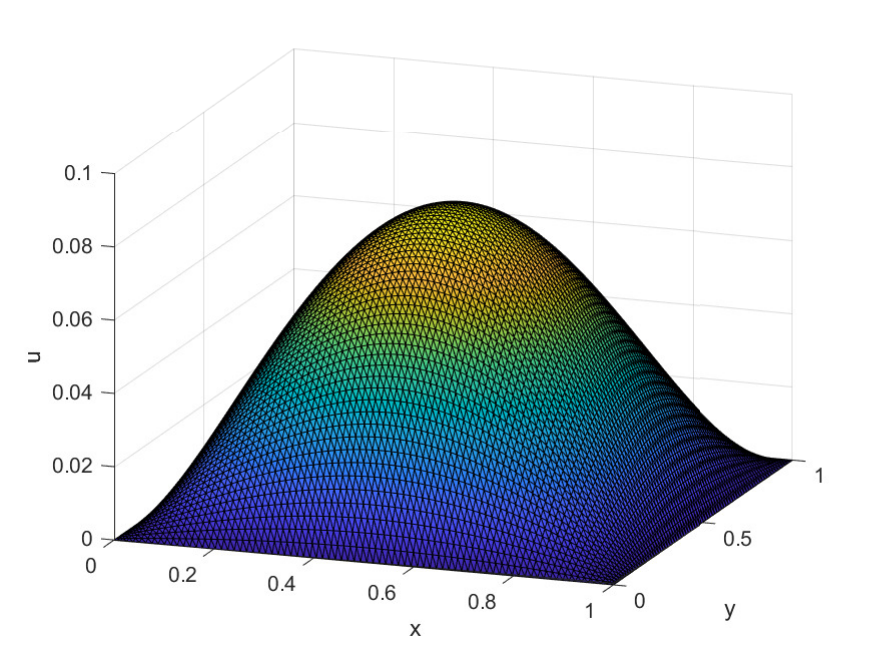}
\includegraphics[width=0.32\textwidth]{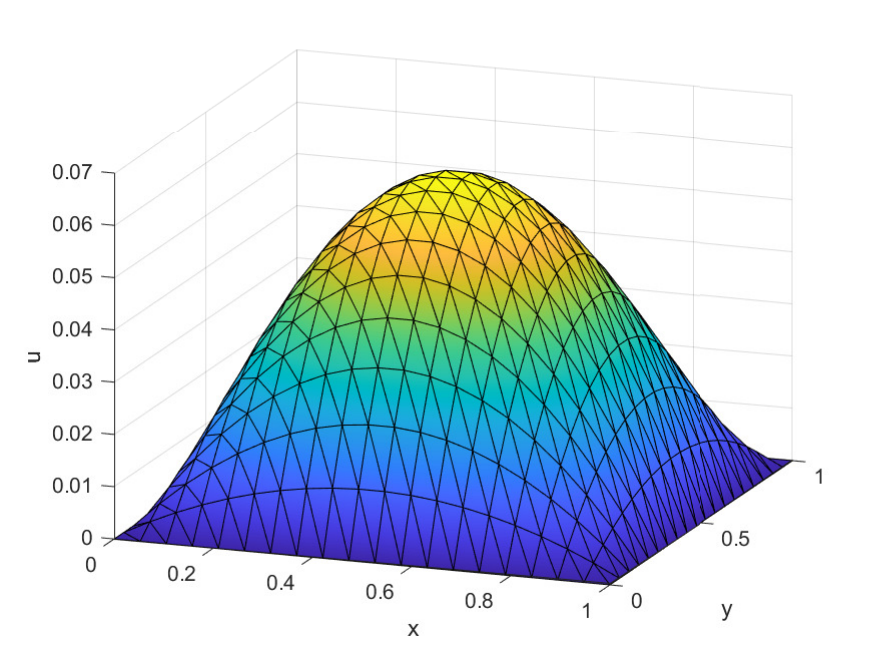} 
\includegraphics[width=0.32\textwidth]{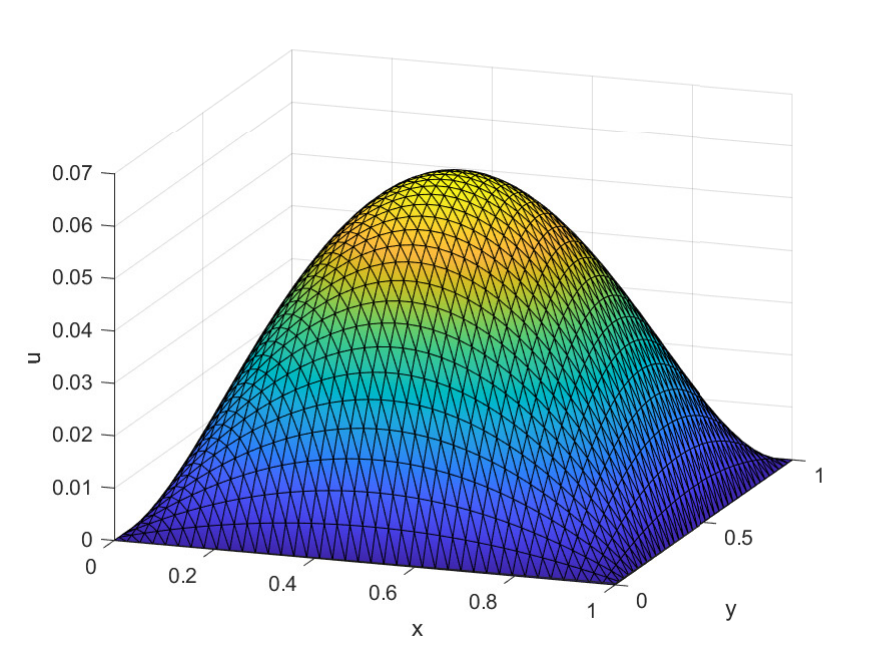} 
\includegraphics[width=0.32\textwidth]{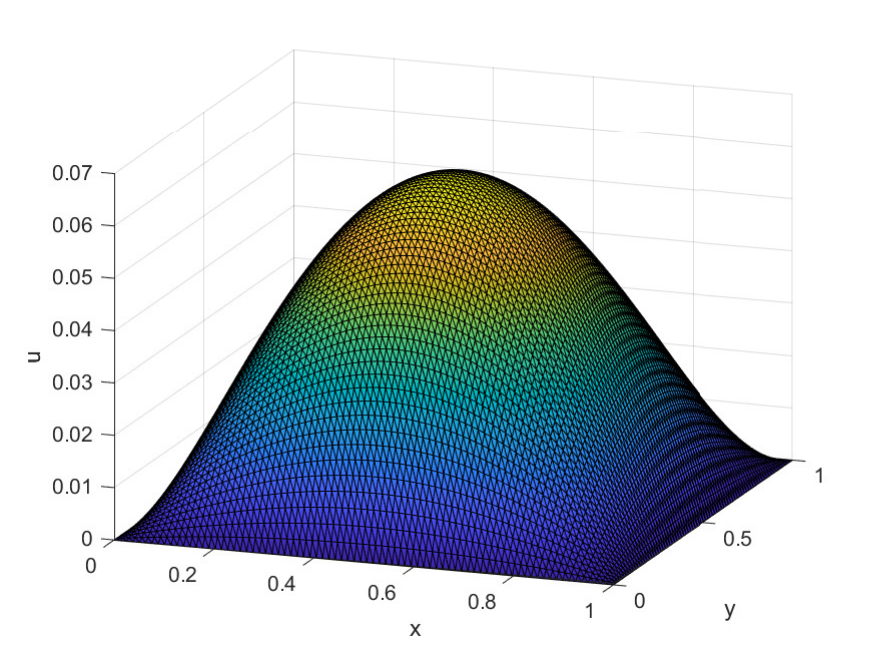}
 \caption{From top to bottom, plot of the reconstruction of the functions $u_i(\mathbf{x})$ using the enriched finite element method based on the set $\mathcal{E}_{10}$, with respect to  the third (left), fourth (center) and fifth (right), levels of refinement $i=1,2,3,4$.}
 \label{plotrec}
\end{figure}

\begin{figure}
  \centering
\includegraphics[width=0.49\textwidth]{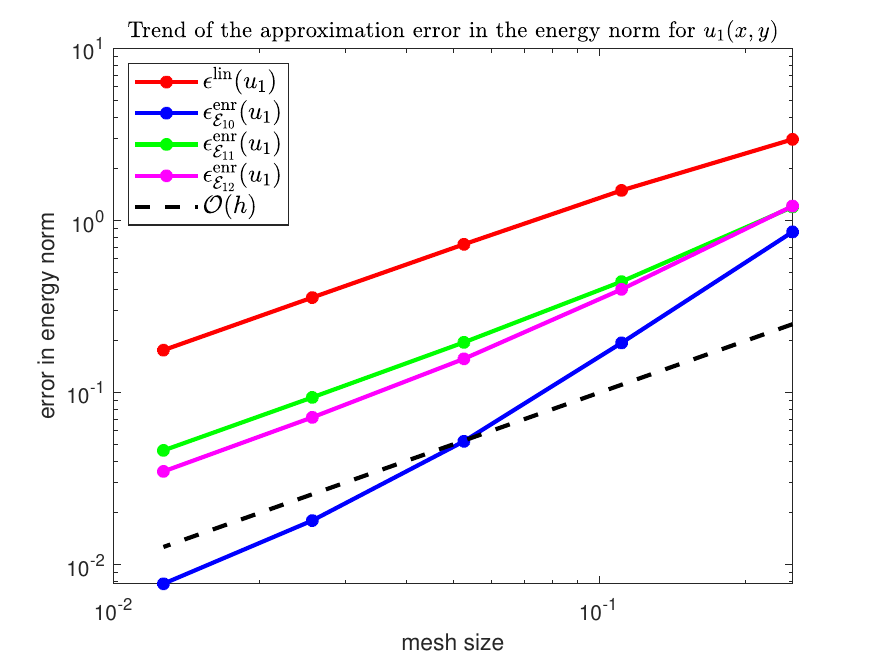} 
\includegraphics[width=0.49\textwidth]{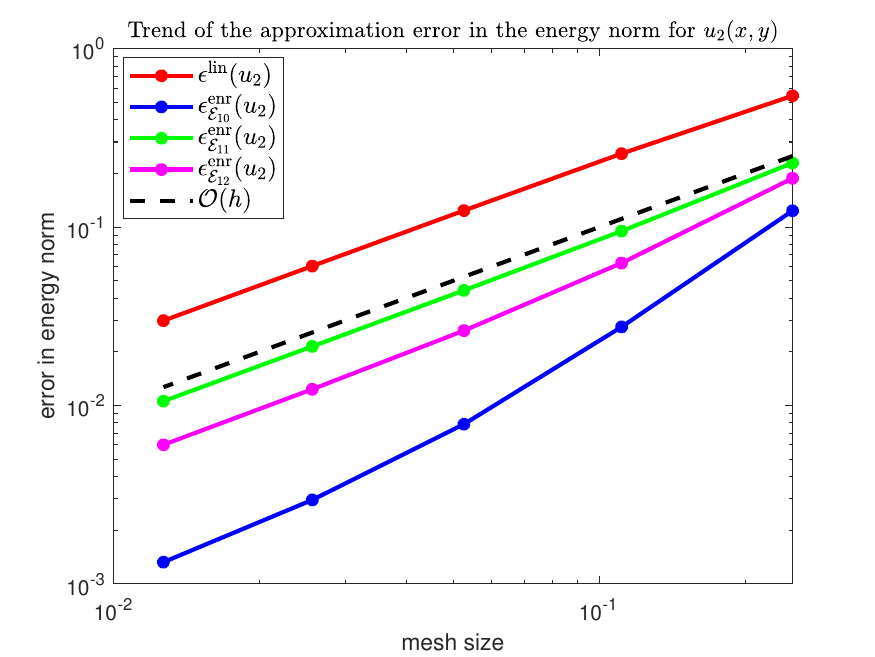} 
 \caption{Trend of the error in the energy norm when approximating the functions $u_1(\mathbf{x})$ (left) and $u_2(\mathbf{x})$ (right) using the linear lagrangian finite element (red) and the enriched finite elements based on the sets of admissible enrichment functions  $\mathcal{E}_{10}$ (blue), $\mathcal{E}_{11}$ (green) and $\mathcal{E}_{12}$ (magenta) with $\mu=\alpha=\beta=0$.}
 \label{Fig1}
\end{figure}

\begin{figure}
  \centering
\includegraphics[width=0.49\textwidth]{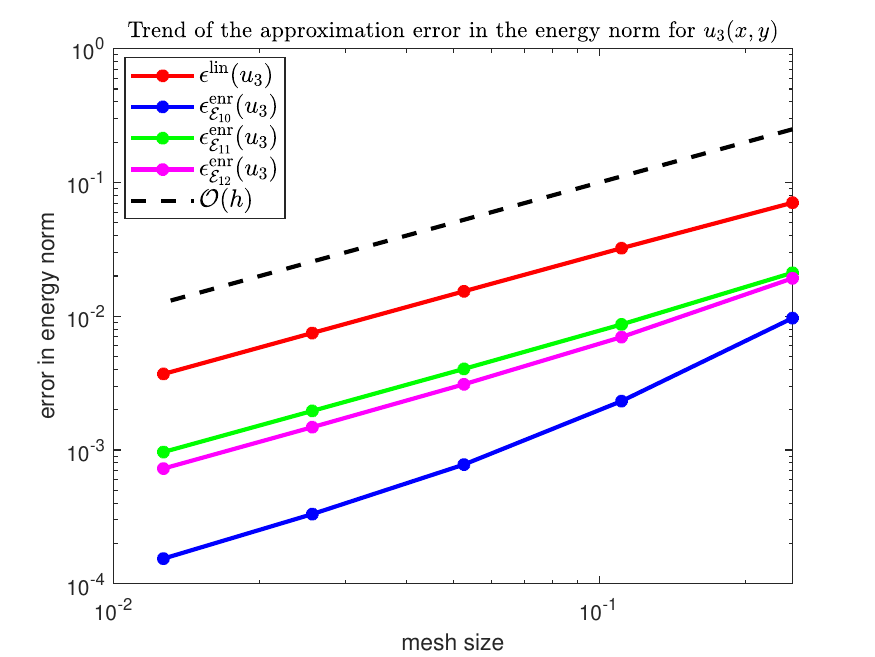} 
\includegraphics[width=0.49\textwidth]{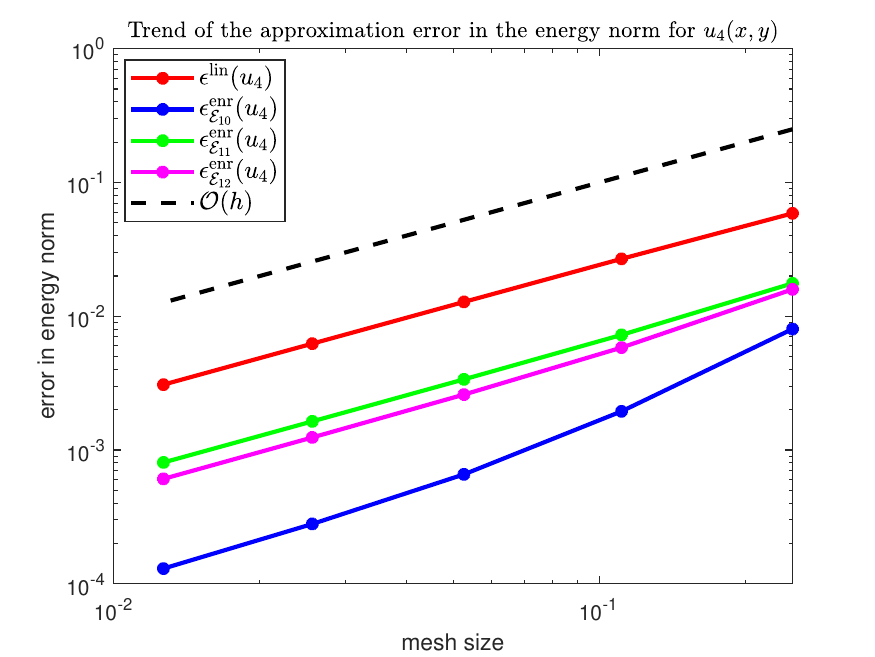} 
 \caption{Trend of the error in the energy norm when approximating the functions $u_3(\mathbf{x})$ (left) and $u_4(\mathbf{x})$ (right) using the linear lagrangian finite element (red) and the enriched finite elements based on the sets of admissible enrichment functions  $\mathcal{E}_{10}$ (blue), $\mathcal{E}_{11}$ (green) and $\mathcal{E}_{12}$ (magenta) with $\mu=\alpha=\beta=0$.}
 \label{Fig2}
\end{figure}

\begin{figure}
  \centering
\includegraphics[width=0.49\textwidth]{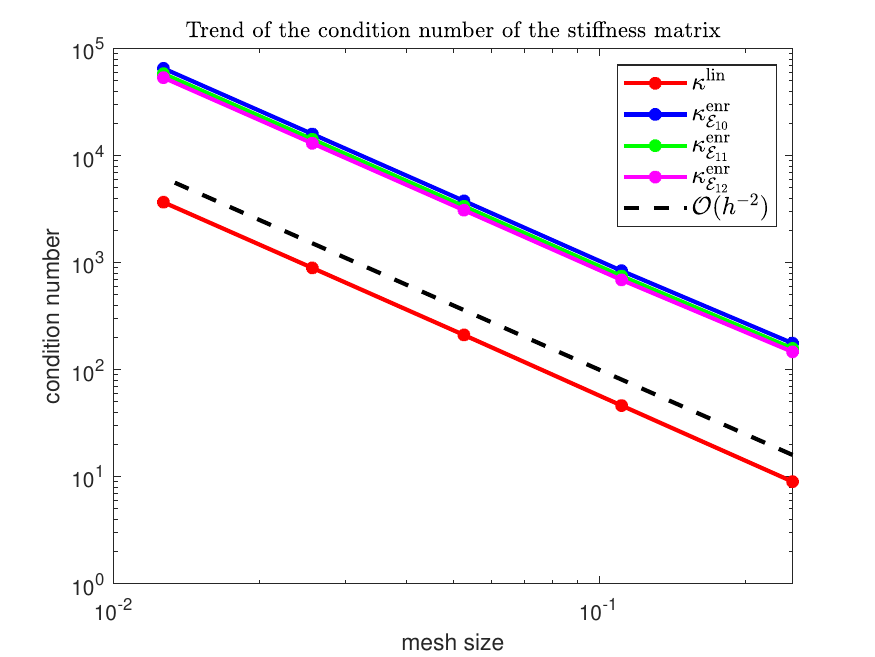} 
 \caption{Trend of the condition number of the stiffness matrix for the linear lagrangian finite element (red) and for the enriched finite elements based on the sets of admissible enrichment functions  $\mathcal{E}_{10}$ (blue), $\mathcal{E}_{11}$ (green) and $\mathcal{E}_{12}$ (magenta) with $\mu=\alpha=\beta=0$.}
 \label{condplot}
\end{figure}

\begin{figure}
  \centering
\includegraphics[width=0.49\textwidth]{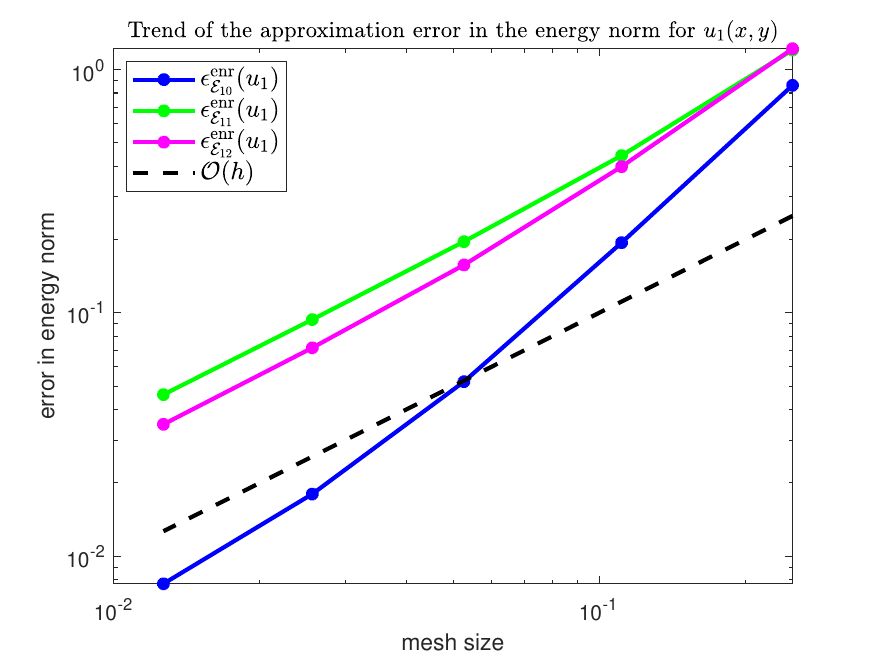} 
\includegraphics[width=0.49\textwidth]{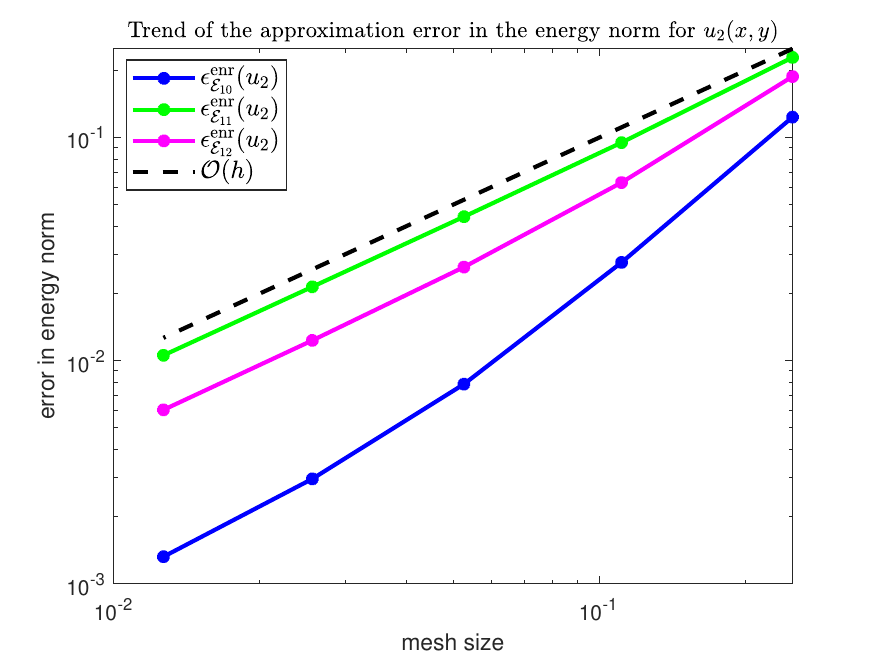} 
 \caption{Trend of the error in the energy norm when approximating the functions $u_1(\mathbf{x})$ (left) and $u_2(\mathbf{x})$ (right) using the enriched finite elements based on the sets of admissible weighted enrichment functions  $\mathcal{E}_{10}$ (blue), $\mathcal{E}_{11}$ (green) and $\mathcal{E}_{12}$ (magenta) with $\mu=\beta=0$ and $\alpha=1$.}
 \label{Fig1w}
\end{figure}

\begin{figure}
  \centering
\includegraphics[width=0.49\textwidth]{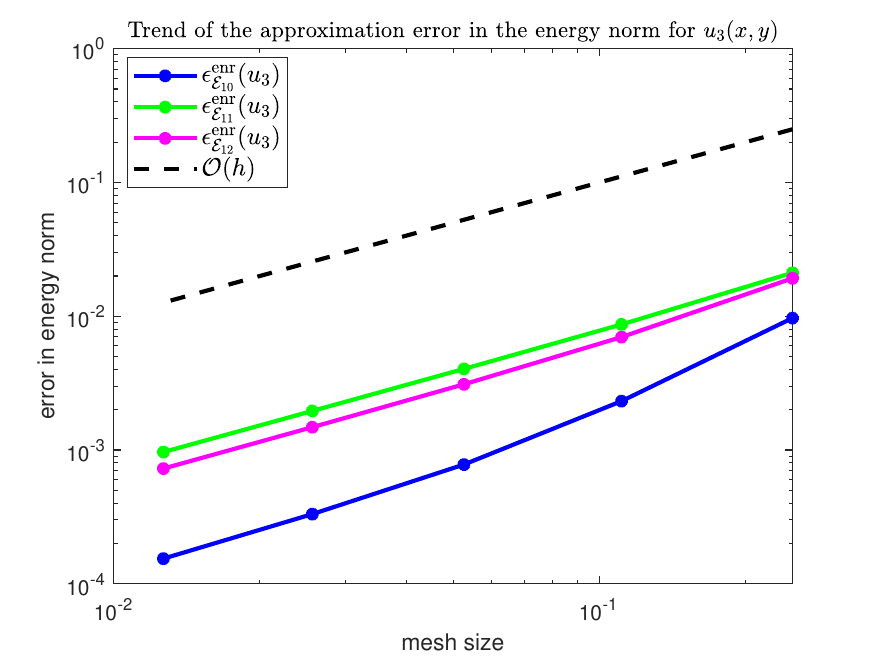} 
\includegraphics[width=0.49\textwidth]{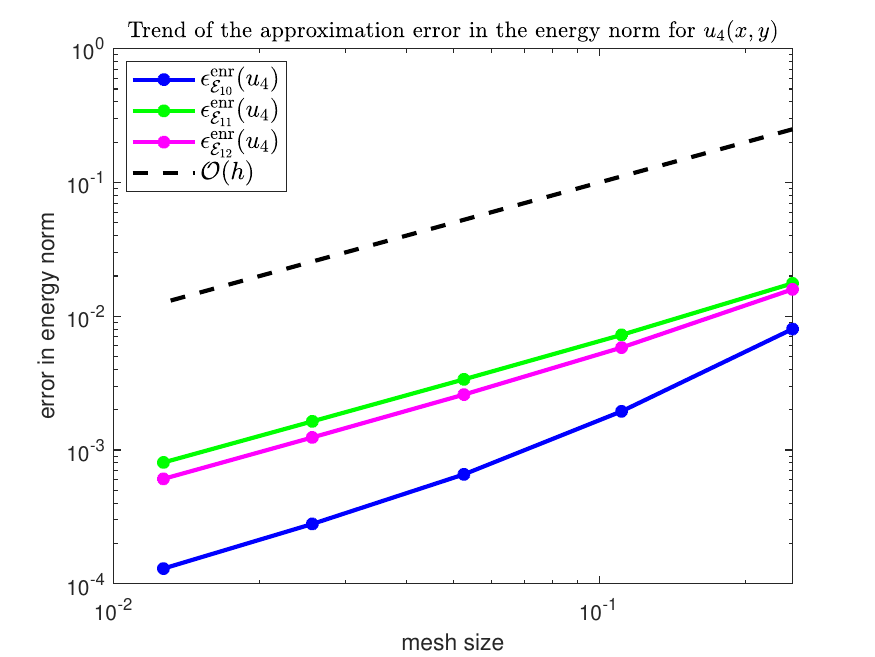} 
 \caption{Trend of the error in the energy norm when approximating the functions $u_3(\mathbf{x})$ (left) and $u_4(\mathbf{x})$ (right) using the enriched finite elements based on the sets of admissible weighted enrichment functions  $\mathcal{E}_{10}$ (blue), $\mathcal{E}_{11}$ (green) and $\mathcal{E}_{12}$ (magenta) with $\mu=\beta=0$ and $\alpha=1$.}
 \label{Fig2w}
\end{figure}

\section{Conclusions and Future Works}
In this paper, we proposed weighted enrichment strategies to enhance the accuracy of the standard
linear finite element method for solving the Poisson problem with Dirichlet boundary conditions. Specifically, we introduced new three-parameter families of weighted enrichment functions and, for one of these families, we established an explicit error bound in $L^2$-norm. Numerical experiments confirmed that the proposed approach significantly enhances the approximation accuracy compared to the linear lagrangian finite element.

\noindent
As future work, we plan to extend this enrichment strategy to the numerical solution of other classes of PDEs, such as those arising in fluid dynamics, elasticity, and wave propagation.

\section*{Acknowledgments}
\noindent
This research has been achieved as part of RITA \textquotedblleft Research
 ITalian network on Approximation'' and as part of the UMI group \enquote{Teoria dell'Approssimazione
 e Applicazioni}. The research was supported by GNCS-INdAM 2025 project \emph{``Polinomi, Splines e Funzioni Kernel: dall'Approssimazione Numerica al Software Open-Source''} and by GNCS-INdAM 2025 project \emph{``High-order BEM based numerical techniques for wave propagation problems''} (CUP E53C24001950001).

\section*{Funding}
\noindent
Project funded by the EuropeanUnion – NextGenerationEU under the National Recovery and Resilience Plan (NRRP), Mission 4 Component 2 Investment 1.1 - Call PRIN 2022 No. 104 of February 2, 2022 of Italian Ministry of University and Research; Project 2022FHCNY3 (subject area: PE - Physical Sciences and Engineering) \enquote{Computational mEthods for Medical Imaging (CEMI)}. \\

\noindent
Luca Desiderio would like to express his gratitude to the financial support provided by the European Union - Next GenerationEU, in the framework of the project ``Strategie HPC e modelli fisico-numerici per la previsione di eventi meteorologici estremi'' (HPC-XTREME) from the National Recovery and Resilience Plan, Mission 4  ``Istruzione e ricerca'' Component 2 ``Dalla ricercar all'impresa'' - Investment 1.4 - NATIONAL CENTER FOR HPC, BIG DATA AND QUANTUM COMPUTING (Project Code CN00000013 - CUP B83C22002830001). The views and opinions expressed are solely those of the authors and do not necessarily reflect those of the European Union, nor can the European Union be held responsible for them.

\bibliographystyle{elsarticle-num}
\bibliography{bibliografia.bib}

\end{document}